\documentclass[lettersize,journal,twoside]{IEEEtran}
\usepackage{amsmath,amsfonts,amssymb}
\usepackage{amsthm}
\usepackage{algpseudocode}
\usepackage{algorithmicx} 
\usepackage{algorithm}
\usepackage{array}
\usepackage[caption=false,font=normalsize,labelfont=sf,textfont=sf]{subfig}
\usepackage{textcomp}
\usepackage{stfloats}
\usepackage{url}
\usepackage{verbatim}
\usepackage{graphicx}
\usepackage{caption}
\usepackage{subfig}
\usepackage{cite}
\usepackage{multirow}
\usepackage{float}
\usepackage{xcolor}
\usepackage{makecell}
\newtheorem{definition}{Definition}
\newtheorem{lemma}{Lemma}

\newtheorem{theorem}{Theorem}

\newtheorem{proposition}{Proposition}

\newcommand{\st}{\text{s.t.}}

\title{A Hybrid Submodular Optimization Approach to Controlled Islanding with Post-Disturbance Stability Guarantees}

\author{Luyao Niu$^{1*}$, Dinuka Sahanbandu$^{1*}$, Andrew Clark$^2$,~\IEEEmembership{Member,~IEEE,} and Radha Poovendran$^1$,~\IEEEmembership{Fellow,~IEEE}
\thanks{This work was supported by the Air Force Office of Scientific Research (AFOSR) through grants FA9550-23-1-0208, FA9550-22-1-0054, and FA9550-20-1-0074.}%
\thanks{*Authors contributed equally to this work.}
\thanks{$^{1}$Luyao Niu, Dinuka Sahabandu, and Radha Poovendran are with the Network Security Lab, Department of Electrical and Computer Engineering, University of Washington, Seattle, WA 98195-2500
        {\tt\small \{luyaoniu,sdinuka,rp3\}@uw.edu}}%
\thanks{$^{2}$Andrew Clark is with the Department of Electrical and Systems Engineering at Washington University in St. Louis, St Louis, MO 63130-4899
        {\tt\small \{andrewclark\}@wustl.edu}}%
}

\begin{document}

\maketitle

\begin{abstract}
Disturbances may create cascading failures in power systems and lead to widespread blackouts.
Controlled islanding is an effective approach to mitigate cascading failures by partitioning the power system into a set of disjoint islands. 
To retain the stability of the power system following disturbances, the islanding strategy should not only be minimally disruptive, but also guarantee post-disturbance stability. 
In this paper, we study the problem of synthesizing post-disturbance stability-aware controlled islanding strategies.
To ensure post-disturbance stability, our computation of islanding strategies takes load-generation balance and transmission line capacity constraints into consideration, leading to a hybrid optimization problem with both discrete and continuous variables.
To mitigate the computational challenge incurred when solving the hybrid optimization program, we propose the concepts of hybrid submodularity and hybrid matroid.
We show that the islanding problem is equivalent to a hybrid matroid optimization program, whose objective function is hybrid supermodular.
Leveraging the supermodularity property, we develop an efficient local search algorithm and show that the proposed algorithm achieves $\frac{1}{2}$-optimality guarantee.
We compare our approach with a baseline using mixed-integer linear program on IEEE 118-bus, IEEE 300-bus, ActivSg 500-bus, and Polish 2383-bus systems. 
Our results show that our approach outperforms the baseline in terms of the total cost incurred during islanding across all test cases. 
Furthermore, our proposed approach can find an islanding strategy for large-scale test cases such as Polish 2383-bus system, whereas the baseline approach becomes intractable.
\end{abstract}


\section{Introduction}
\IEEEPARstart{P}{ower} systems are often operated close to their stability and capacity limits.
When power systems incur disturbances such as cyber attacks \cite{case2016analysis}, natural disasters \cite{busby2021cascading}, and spiking growth in demand \cite{pidd2012india}, some transmission lines and generators may experience outages, leading to overflow at neighboring transmission lines and hence cascading failures \cite{hines2009cascading}.
Cascading failures can lead to disastrous consequences. 
For example, the 2023 Pakistan blackout \cite{pakistan2023blackout} left approximately $220$ million people without power.


One approach to improve the resilience of power system and mitigate cascading failures is through \emph{controlled islanding} \cite{kamali2020controlled}.
Controlled islanding deliberately trips a subset of transmission lines to partition the power system into a collection of disjoint, internally stable and connected islands.
Currently, metrics such as power flow disruption, generator coherency, and load-generation imbalance are widely adopted when computing controlled islanding strategies \cite{yusof1993slow,you2004slow,xu2009slow,kyriacou2017controlled,patsakis2019strong,teymouri2019milp,trodden2013optimization,trodden2013milp,kamali2020intentional,esmaili2020convex}.

In addition to the aforementioned metrics, the post-disturbance stability needs to be taken into consideration when computing the controlled islanding strategies.
For example, if a controlled islanding strategy cannot guarantee the post-islanding power flow to be within the capacity of each transmission line, controlled islanding may have to be executed again within one newly formulated island to avoid cascading failure inside the island.
In the worst-case, such islanding strategies may have to be executed multiple times, leading to widely-spread blackout.
Furthermore, the controlled islanding strategy should partition the power system in a way such that the tripped connections can be re-established and thus the whole power system can be restored efficiently \cite{sarmadi2011sectionalizing,patsakis2019strong,tortos2012controlled}.
To allow efficient restoration, each island is required to have sufficient availability to blackstart generators. 

At present, computing a controlled islanding strategy that jointly (i) optimizes the widely-adopted metrics such as power flow disruption and generator coherency, (ii) takes post-disturbance stability into consideration, and (iii) guarantees restoration after islanding,
has been less studied.
Synthesizing such an islanding strategy involves both continuous variables (post-islanding power flow and load shedding at each load bus for load-generation balance) and discrete variables (choices of tripping transmission lines and allocation of blackstart generators).
Such hybrid optimization programs are generally NP-hard. 
Finding exact solutions to these programs is computationally intensive, and hence does not scale to large-scale power systems.
Although there exist heuristic solution algorithms to compute controlled islanding strategies \cite{zhao2003study,sun2003splitting,wang2010novel}, they omit some constraints from (i)-(iii) and cannot provide provable optimality guarantees.

In this paper, we investigate how to compute a controlled islanding strategy that jointly satisfies three goals. 
The first goal is to optimize metrics including power flow disruption and generator coherency.
Second, the islanding strategy considers post-disturbance stability, which is captured by satisfying load-generation balance to each island and ensuring the post-islanding power flow on each transmission line to remain with capacity limit.
The third goal of the islanding strategy is to guarantee efficient restoration of power system after islanding by ensuring sufficient availability of energized blackstart generators within each island.
We formulate a hybrid optimization to compute such a controlled islanding strategy.
To solve the hybrid optimization problem, we propose a concept named \emph{hybrid submodularity} as a generalization of discrete submodularity. 
We prove that the objective functions and constraints are hybrid submodular, and develop an efficient algorithm to compute the islanding strategy with provable optimality guarantees.
To summarize, this paper makes the following contributions.
\begin{itemize}
    \item We formulate the problem of synthesizing a controlled islanding strategy that jointly optimizes power flow disruption, dynamical stability, post-disturbance stability, and blackstart generator allocation for restoration. 
    \item We translate the islanding problem to a matroid optimization problem. 
    We prove that transmission line capacity constraint for post-islanding power flow and blackstart generator allocation can be encoded by hybrid monotone and supermodular functions.
    \item We present a local search algorithm to compute the islanding strategy.
    We develop a $\frac{1}{2}$-optimality bound for the algorithm based on hybrid submodularity property.
    \item We evaluate our proposed approach on IEEE 118-bus, IEEE 300-bus, ActivSg 500-bus, and Polish 2383-bus systems. Compared with a baseline approach using mixed-integer linear program (MILP)-based formulation, our results show that the proposed solution approach outperforms the baseline in terms of both total cost and amount of load shedding. 
    Moreover, our approach scales well to large-scale test cases whereas the baseline approach becomes infeasible.
\end{itemize}

The present paper generalizes the approach proposed in \cite{sahabandu2022hybrid} in the following aspects. 
This paper incorporates both the post-disturbance stability and availability to blackstart generators in each island when computing controlled islanding strategies, which are not considered in \cite{sahabandu2022hybrid}.
We prove that both constraints on transmission line capacity and blackstart generator allocation are hybrid supermodular in the islanding strategies, and develop an efficient solution algorithm with $\frac{1}{2}$-optimality guarantee.
We implement the proposed solution approach to large-scale power systems, namely IEEE 118-bus, IEEE 300-bus, ActivSg 500-bus, and Polish 2383-bus systems.

The remainder of this paper is organized as follows.
Section \ref{sec:related} reviews related literature. 
In Section \ref{sec:background}, we present preliminary background on submodularity and matroids.
Section \ref{sec:formulation} formulates the controlled islanding problem.
Section \ref{sec:sol} translates the islanding problem to a matroid optimization problem, and develops a local search algorithm based on hybrid submodularity property with $\frac{1}{2}$-optimality guarantee.
In Section \ref{sec:experiment}, we evaluate our proposed approach using IEEE 118-bus, IEEE 300-bus, ActivSg 500-bus, and Polish 2383-bus systems.
Section \ref{sec:conclusion} concludes the paper.

\section{Related Work}\label{sec:related}

Power systems are often operated at a stringent operating point. 
Therefore, disturbances such as cyber attacks \cite{case2016analysis} and natural disasters \cite{busby2021cascading} could destablize power systems and cause cascading failures.
Controlled islanding has been shown to be an effective approach to mitigate cascading failures following disturbances \cite{you2004slow, sun2003splitting, ding2012two}.

A widely adopted solution to computing controlled islanding strategy is based on slow coherency \cite{yusof1993slow,you2004slow,xu2009slow}. 
These category of approaches groups generators using slow coherency analysis.
Then an islanding strategy can be found by partitioning the power system so that non-coherent generators are disconnected.
Slow coherency analysis-based computation of islanding strategies requires slow eigenbasis analysis, and may not scale to power systems of large sizes.

Spectral clustering controlled islanding \cite{ding2012two} has been proposed to improve the computation efficiency to obtain the islanding strategies.
This approach follows two steps, where the first step groups the generators based on their dynamics, and the second step trips the transmission lines using metrics such as power flow disruption or load-generation imbalance.
An alternative class of approaches to improve the scalability of controlled islanding algorithms is to applying ordered binary decision diagram (OBDD) methods on simplified graph representations of the power systems \cite{zhao2003study,sun2003splitting}. 
Such approaches may overly simplify the power systems, and do not provide provable optimality guarantees on the obtained islanding strategies.

To obtain the exact islanding strategies, mixed-integler linear programs have been formulated to incorporate different constraints such as power flow disruption, load-generation imbalance, generator coherency, and power system restoration \cite{kyriacou2017controlled,patsakis2019strong,teymouri2019milp,trodden2013optimization,trodden2013milp,kamali2020intentional,esmaili2020convex}.
However, solving NP-hard MILPs formulated on large-scale power systems is computationally expensive.
A linear program-based microgrid formation is investigated in \cite{pang2022formulation}. 
However, this approach is network dependent, and may not be readily generalized to arbitrary large-scale power systems.

Submodularity-based approaches \cite{liu2018controlled, sahabandu2022hybrid,sahabandu2022submodular} have recently been proposed to not only efficiently compute controlled islanding strategies for large-scale power systems, but also provide provable optimality guarantees for the obtained islanding strategies. 
In \cite{sahabandu2022submodular}, metrics including power flow disruption, generator coherency, and load-generation imbalance are shown to be monotone and submodular, leading to efficient local search solution algorithms with $\frac{1}{2}$ optimality bounds.
Submodular optimization has also been adopted in other application scenarios such as the placement of energy storage units \cite{qin2016submodularity}, voltage control \cite{liu2017submodular}, and distribution network configuration \cite{khodabakhsh2017submodular}.
The concept of hybrid submodularity in this paper can benefit these solutions to incorporate continuous variables such as capacities of storage units into consideration to improve the optimalities of solutions \cite{qin2016submodularity,liu2017submodular,khodabakhsh2017submodular}.
\section{Preliminary Background}\label{sec:background}

This section introduces preliminary background on submodularity and matroids. Consider a finite set $\mathcal{V}$. 
A function $f:2^\mathcal{V}\rightarrow \mathbb
R_{\geq 0}$ is monotone nondecreasing if $f(\mathcal{S})\leq f(\mathcal{T})$ for any $\mathcal{S}\subseteq\mathcal{T}$.
Function $f$ is said to be submodular \cite{fujishige2005submodular} if $$f(\mathcal{S}\cup\{v\})-f(\mathcal{S})\geq f(\mathcal{T}\cup\{v\})-f(\mathcal{T})$$ holds for any $\mathcal{S}\subseteq\mathcal{T}\subseteq\mathcal{V}$ and $v\in\mathcal{V}\setminus \mathcal{T}$.
A function $f$ is supermodular if $-f$ is submodular.

We present the definition for matroids, which give rise to a class of submodular functions, as follows.
\begin{definition}[Matroid]
A matroid $\mathcal{M}$ is a pair $(\mathcal{V},\mathcal{I})$, where $\mathcal{V}$ is a finite set, and $\mathcal{I}\subseteq 2^\mathcal{V}$ is a collection of subsets of $\mathcal{V}$ satisfying (i) $\emptyset\in\mathcal{I}$, (ii) if $\mathcal{S}\subseteq\mathcal{T}\subseteq\mathcal{V}$ and $\mathcal{T}\in\mathcal{I}$, then $\mathcal{S}\in\mathcal{I}$, and (iii) if $\mathcal{S},\mathcal{T}\in\mathcal{I}$ and $|\mathcal{T}|>|\mathcal{S}|$, then there exists $v\in\mathcal{T}\setminus\mathcal{S}$ such that $\mathcal{S}\cup\{v\}\in\mathcal{I}$.
\end{definition}

Each set in $\mathcal{I}$ is called an independent set of matroid $\mathcal{M}$.
A maximal independent set of $\mathcal{M}$ is a basis of $\mathcal{M}$.
The set of bases of matroid $\mathcal{M}$ is denoted as $\mathcal{B}(\mathcal{M})$.
The rank function $\rho_\mathcal{M}:2^\mathcal{V}\rightarrow \mathbb{Z}_{\geq 0}$ is defined as $\rho_\mathcal{M}(\mathcal{T})=\max\{|\mathcal{S}|:\mathcal{S}\subseteq\mathcal{T},\mathcal{S}\in\mathcal{I}\}$.
The rank function of matroid $\mathcal{M}$ is nondecreasing and submodular \cite{oxley2006matroid}.

In this paper, we consider graphic matroid $\mathcal{M}_G$ induced by a graph $G=(\mathcal{V},\mathcal{E})$, where $\mathcal{V}$ is a finite set of vertices, and $\mathcal{E}\subseteq\mathcal{V}\times\mathcal{V}$ is the set of edges.
Graphic matroid $\mathcal{M}_G$ is defined as $\mathcal{M}_G=(\mathcal{E},\mathcal{I})$.
Each independent set in $\mathcal{I}$ of graphic matroid $\mathcal{M}_G$ is acyclic.
For a connected graph $G$, the bases $\mathcal{B}(\mathcal{M}_G)$ are the spanning trees of graph $G$ \cite{oxley2006matroid}.
\section{System Model and Islanding Formulation}\label{sec:formulation}

We consider a power system and model it as a graph, denoted as $G=(\mathcal{V},\mathcal{E})$, where $\mathcal{V}$ is the set of vertices and $\mathcal{E}\subseteq\mathcal{V}\times\mathcal{V}$ is the set of edges. 
Without loss of generality, we assume that graph $G$ is connected.
Here we use the set of vertices $\mathcal{V}=\{1,\ldots,n\}$ to represent the set of buses, and use the set of edges $\mathcal{E}$ to represent the collection of transmission lines.
Among the buses, we denote the set of generator and load buses as $\mathcal{G}\subseteq \mathcal{V}$ and $\mathcal{L}\subseteq \mathcal{V}$, respectively, where $\mathcal{G}\cap\mathcal{L}=\emptyset$.
For each generator bus $j\in\mathcal{G}$, we denote its generation capacity as $\Bar{g}_j$.
The set of generators $\mathcal{G}$ can be classified based on whether a unit has blackstart capability or relies on cranking power from the system to restart.
If a generator bus $j$ is attached with a blackstart generator, we represent it as $\mathbb{I}_j=1$. 
Otherwise $\mathbb{I}_j=0$.
We further group the set of generators $\mathcal{G}$ into $m$ coherence groups based on their coherence \cite{you2004slow}.
We select a reference generator $r_k$ within each coherence group.
For each load bus $j\in\mathcal{L}$, we denote its maximum load as $\Bar{d}_j$. 
For each transmission line $(j,j')\in\mathcal{E}$, we denote the power flow from bus $j$ to $j'$ through transmission line $(j,j')$ as $P_{jj'}$. 
The power flow $P_{jj'}$ should satisfy $P_{jj'}\leq \Bar{P}_{jj'}$ to avoid overflow and transmission line outage, where $\Bar{P}_{jj'}$ is the transmission line capacity of $(j,j')$.


Our goal is to partition the power system $G$ into $m$ disjoint islands, denoted as $I_1,\ldots,I_m$, to mitigate cascading failures.
Each island $I_k$ can be represented as 
\begin{equation*}
    I_k=(\mathcal{V}_k,\mathcal{E}_k),
\end{equation*}
where $\mathcal{V}_k\subset\mathcal{V}$ is a subset of buses and $\mathcal{E}_k=(\mathcal{V}_k\times\mathcal{V}_k)\cap \mathcal{E}$ is a subset of transmission lines.
To ensure the islands to be disjoint, we have that $\mathcal{V}_k\cap\mathcal{V}_{k'}=\emptyset$ for all $k\neq k'$.
Note that for each reference generator $r_k$, it must be contained within island $I_k$ so that $r_k$ and $r_{k'}$ are disconnected for any $k\neq k'$.
In addition, each island $I_k$ is connected for all $k=1,\ldots,m$.
In the remainder of this paper, we will denote $H=\cup_{k=1}^m\mathcal{E}_k$ as the set of transmission lines that remain connected in the power system after formulating the islands.

In this paper, we partition the power system into $m$ number of islands by optimizing generator coherency and power flow disruption. We formulate each metric in detail below.

We denote the coherence matrix as $\Bar{A}\in\mathbb{R}^{m\times |\mathcal{G}|}$ (see Appendix for a detailed derivation of $\Bar{A}$.) The generator coherency is given as
\begin{equation}\label{eq:coherence}
    F_1(H) = \|\Bar{A}-A(H)\|_F,
\end{equation}
where $\|\cdot\|_F$ represents the Frobenius norm, and $A(H)\in\mathbb{R}^{|\mathcal{G}|\times m}$ whose $(j,k)$-th entry is defined as $\left[A(H)\right]_{jk} = 1$ if $j\in \mathcal{V}_k$ and zero otherwise.

Power flow disruption measures the total amount of power flow being disrupted due to tripping transmission lines to form islands.
The power flow disruption can be computed as
\begin{equation}\label{eq:disruption}
    F_2(H) = \sum_{(j,j')\in\mathcal{E}\setminus H}\frac{|P_{jj'}|}{2},
\end{equation}
where $P_{jj'}$ is the pre-islanding power flow. 

In addition to generator coherency and power flow disruption, the formed islands $I_1,\ldots,I_m$ need to guarantee post-disturbance stability.
In particular, we will ensure load-generation balance and transmission line capacity constraints on the post-islanding power flow.

Load-generation balance requires that the amount of power generation within each island should be sufficient to meet the total power demand.
Due to the islanding operation, there may exist deficit in power generation, requiring load shedding from some buses.
We define $c_j:\mathbb{R}\rightarrow\mathbb{R}$ as the cost function of load shedding for load bus $j$.
We assume that $c_j$ is monotone increasing and convex.
Then given set $H$, the loads $d=[d_1,\ldots,d_{|\mathcal{L}|}]^\top$ of buses in $\mathcal{L}$ can be found by
\begin{subequations}\label{eq:island balance}
\begin{align}
    F_3(H,d) = &\min_d\quad\sum_{j\in\mathcal{L}}c_j(\Bar{d}_j-d_j)\\
    &\st\quad\sum_{j\in \mathcal{L}\cap \mathcal{V}_k} d_j\leq \sum_{j\in \mathcal{G}\cap \mathcal{V}_k} \Bar{g}_j,~\forall k=1,\ldots,m\label{eq:island balance-2}\\
    &\quad\quad\quad d_j\in[0,\Bar{d}_j],~\forall j\in\mathcal{L}.\label{eq:island balance-3}
\end{align}
\end{subequations}
Any feasible solution to optimization program \eqref{eq:island balance} guarantees load-generation balance.
In the remainder of this paper, we define $(H,d)$ as an islanding strategy, which specifies the set of transmission lines that will remain in the system as well as the amount of load shedding incurred by the load buses.


We next focus on the transmission line capacity constraint for post-islanding power flow, denoted as $\Tilde{P}_{jj'}$.
The post-islanding power flow should satisfy the conservation law
\begin{equation}\label{eq:conservation}
    \sum_{j':(j,j')\in\mathcal{E}}\Tilde{P}_{jj'}+ d_j - g_j = 0,~\forall j\in\mathcal{V}.
\end{equation}
Furthermore, $\Tilde{P}_{jj'}$ needs to satisfy the transmission line capacity constraint given as below
\begin{equation}\label{eq:island capacity}
    \Tilde{P}_{jj'}\leq \Bar{P}_{jj'}, ~(j,j')\in H,
\end{equation}
in order to ensure the post-disturbance stability.

To allow the power system to be restored, each island should contain at least one blackstart generator. 
This constraint is formulated as follows
\begin{equation}\label{eq:blackstart}
    \sum_{j\in\mathcal{V}_k}\mathbb{I}_{j}\geq 1,~\forall k=1,\ldots,m.
\end{equation}

Let $\alpha_i\in(0,1)$ for $i=1,2,3$ be constants modeling the trade-off parameters among metrics $F_1(H)$, $F_2(H)$, and $F_3(H,d)$, respectively. 
We then formulate the following optimization program to compute the islanding strategy $(H,d)$
\begin{subequations}\label{eq:islanding formulation}
\begin{align}
    \min_{E,d} \quad&\alpha_1F_1(H)+\alpha_2F_2(H) + \alpha_3F_3(H,d)\\
    \st \quad & r_k\in I_k,~\forall k=1,\ldots,m\label{eq:ref gen}\\
    &\text{$I_1,\ldots,I_m$ form a set of islands}\label{eq:islanding validity constraint}\\
    &\text{Eqn.  \eqref{eq:island balance-2}, \eqref{eq:island balance-3}, \eqref{eq:conservation}, \eqref{eq:island capacity}, and \eqref{eq:blackstart}}\label{eq:islanding constraint}
\end{align}
\end{subequations}

\section{Hybrid Submodularity and Proposed Solution}\label{sec:sol}

Solving optimization problem \eqref{eq:islanding formulation} is computationally intractable for large-scale power systems since it involves coupled continuous and discrete variables. Furthermore, the solution space of the discrete variable $H$ grows exponentially with respect to the size of graph $G$.

In this section, we present a computationally efficient solution approach to compute an islanding strategy with provable optimality guarantee. 
We first propose a concept named hybrid submodularity. 
Next, we  prove that the metrics ($F_1(H)$, $F_2(H)$, and $F_3(H,d)$) and the constraints (Eqn. \eqref{eq:island balance-2}, \eqref{eq:island balance-3}, \eqref{eq:island capacity}, and \eqref{eq:blackstart}) satisfy hybrid submodular properties. 
We finally present a local search algorithm to compute the islanding strategy and prove the optimality guarantee based on the hybrid submodular property. 

\subsection{Hybrid Submodularity}
This subsection introduces the concepts of hybrid monotonicity and hybrid submodularity.
\begin{definition}[Hybrid Monotonocity]
Let $\mathcal{V}$ be a finite set and $\mathcal{D}\subset\mathbb{R}^n$. A function $f:2^\mathcal{V}\times \mathcal{F}(\mathcal{D})\rightarrow \mathbb{R}$ is hybrid monotone nondecreasing if, for any finite sets $\mathcal{S}\subseteq\mathcal{T}\subseteq\mathcal{V}$ and $\Lambda\subseteq\Lambda'\subseteq\mathcal{D}$, we have 
\begin{equation}
    f(\mathcal{S},\Lambda)\leq f(\mathcal{T},\Lambda'),
\end{equation}
where $\mathcal{F}(\mathcal{D})$ represents the collection of finite subsets of $\mathcal{D}$.
A function $f$ is hybrid monotone nonincreasing if $-f$ is hybrid monotone nondecreasing.
\end{definition}
We next define hybrid submodularity which generalizes the classic submodularity property defined over discrete sets by incorporating infinite ground set.
\begin{definition}[Hybrid Submodularity]\label{def:hybrid submodular}
Let $\mathcal{V}$ be a finite set and $\mathcal{D}\subset\mathbb{R}^n$. A function $f:2^\mathcal{V}\times \mathcal{F}(\mathcal{D})\rightarrow \mathbb{R}$ is hybrid submodular if, for any $\mathcal{S}\subseteq\mathcal{T}\subseteq\mathcal{V}$ and $\Lambda\subseteq\Lambda'\subseteq\mathcal{D}$, the following properties hold:
\begin{enumerate}
    \item For any $j\in\mathcal{V}\setminus\mathcal{T}$, we have
    \begin{equation*}
        f(\mathcal{S}\cup\{j\},\Lambda)-f(\mathcal{S},\Lambda)\geq f(\mathcal{T}\cup\{j\},\Lambda')-f(\mathcal{T},\Lambda').
    \end{equation*}
    \item For any $\lambda\in \mathcal{D}\setminus\Lambda'$, we have
    \begin{equation*}
        f(\mathcal{S},\Lambda\cup\{\lambda\})-f(\mathcal{S},\Lambda)\geq f(\mathcal{T},\Lambda'\cup\{\lambda\})-f(\mathcal{T},\Lambda').
    \end{equation*}
\end{enumerate}
\end{definition}
We say a function $f$ is hybrid supermodular if $-f$ is hybrid submodular.
We finally present hybrid matroid as a generalization of matroids to discrete sets.
\begin{definition}[Hybrid Matroid]\label{def:hybrid matroid}
A hybrid matroid is a tuple $\mathcal{M}=(\mathcal{V},\mathcal{D},\mathcal{I})$, where $\mathcal{V}$ is a finite set, $\mathcal{D}\subseteq\mathbb{R}^n$, and $\mathcal{I}$ is a collection of subsets of $\mathcal{V}\times\mathcal{D}$ such that
\begin{enumerate}
    \item $(\emptyset,\emptyset)\in\mathcal{I}$.
    \item if $(\mathcal{T},\Lambda')\in\mathcal{I}$, $\mathcal{S}\subseteq\mathcal{T}$, and $\Lambda\subseteq\Lambda'$, then $(\mathcal{S},\Lambda)\in\mathcal{I}$.
    \item if $(\mathcal{S},\Lambda),(\mathcal{T},\Lambda')\in\mathcal{I}$, where $|\mathcal{S}|<|\mathcal{T}|$ and $|\Lambda|<|\Lambda'|$, then there exists $v\in\mathcal{T}\setminus\mathcal{S}$ and $\lambda\in\Lambda'\setminus\Lambda$ such that $(\mathcal{S}\cup\{v\},\lambda\cup\{\lambda\})\in\mathcal{I}$.
\end{enumerate}
\end{definition}
If set $(\mathcal{S},\Lambda)\in\mathcal{I}$, we then say $(\mathcal{S},\Lambda)$ is an indpendent set of hybrid matroid $\mathcal{M}$.
A maximal independent set of $\mathcal{M}$ is a basis of $\mathcal{M}$.
We denote the set of bases of $\mathcal{M}$ as $\mathcal{B}(\mathcal{M})$.

Let a function $f$ be hybrid monotone nondecreasing and submodular and $\mathcal{M}$ be a hybrid matroid. 
Then we can derive a local optimality guarantee for function $f$, as stated below.
\begin{proposition}[\cite{sahabandu2022hybrid}]\label{prop:optimality}
Suppose that function $f:2^\mathcal{V}\times \mathcal{F}(\mathcal{D})\rightarrow\mathbb{R}$ is hybrid monotone nondecreasing and submodular. Let $\mathcal{M}$ be a hybrid matroid and $(\mathcal{S},\Lambda)\in\mathcal{B}(\mathcal{M})$. 
For any $v\in\mathcal{S}$, $t\notin \mathcal{S}$, $\lambda\in\Lambda$, and $\lambda'\notin\Lambda$ satisfying $(\mathcal{S}\setminus\{v\}\cup\{t\},\Lambda\setminus\{\lambda\}\cup\{\lambda'\})$, if we have
\begin{equation*}
    f(\mathcal{S},\Lambda)\geq f(\mathcal{S}\setminus\{v\}\cup\{t\},\Lambda\setminus\{\lambda\}\cup\{\lambda'\}),
\end{equation*}
then for any $(\mathcal{T},\Lambda')\in\mathcal{B}(\mathcal{M})$, $f(\mathcal{S},\Lambda)\geq \frac{1}{2}f(\mathcal{T},\Lambda')$.
\end{proposition}

\subsection{Reformulating Eqn. \eqref{eq:islanding formulation} as Hybrid Matroid Optimization}

This subsection relaxes the formulation in Eqn. \eqref{eq:islanding formulation}, and converts the problem to a hybrid matroid optimization program. 
We show that the objective function of the hybrid matroid optimization program is hybrid supermodular.

We augment the power system $G$ to $\Bar{G}=(\Bar{\mathcal{V}},\Bar{\mathcal{E}})$, where $\Bar{\mathcal{V}}=\mathcal{V}\cup\{a\}$ is obtained by introducing a supernode $a$, and 
\begin{equation}
    \Bar{\mathcal{E}}=\mathcal{E}\cup\{(a,r_k):k=1,\ldots,m\}.
\end{equation}
The set of edges $\{(a,r_k):k=1,\ldots,m\}$ connects the supernode $a$ with each reference generator $r_i$.

Since the reference generators belong to different coherence groups, they should be partitioned into different islands, i.e., there exist no $r_k$ and $r_{k'}$ such that $r_k,r_{k'}\in \mathcal{V}_k$ for any $k\neq k'$.
Based on this insight, we have the following preliminary result.
\begin{lemma}\label{lemma:spanning tree}
If set $H\cup\{(a,r_k):k=1,\ldots,m\}$ forms a spanning tree of $\Bar{G}$, then islands $I_1,\ldots,I_m$ induced by $H$ on power system $G$ satisfy constraints \eqref{eq:ref gen} and \eqref{eq:islanding validity constraint}.
\end{lemma}
\begin{proof}
We prove the lemma by contradiction. 
Suppose that set $H\cup\{(a,r_k):k=1,\ldots,m\}$ forms a spanning tree of $\Bar{G}$ while there exist reference generators $r_k,r_{k'}\in \mathcal{V}_{k''}$ for some $k\neq k'$.
Since $r_k,r_{k'}\in \mathcal{V}_{k''}$ and island $I_{k''}$ is connected to be an island, we have that there exists some path $r_k,j_1,\ldots,j_t,r_{k'}$ that connects reference generators $r_k$ and $r_{k'}$ without going through supernode $a$, i.e., $j_l\neq a$ for all $l=1,\ldots,t$.
By the construction of graph $\Bar{G}$, reference generators $r_k$ and $r_{k'}$ are also connected by another path $r_k,a,r_{k'}$ that goes through supernode $a$.
Hence, paths $r_k,a,r_{k'}$ and $r_k,j_1,\ldots,j_t,r_{k'}$ form a cycle, which contradicts the definition of spanning tree. 
Therefore, we have that if set $H\cup\{(a,r_k):k=1,\ldots,m\}$ forms a spanning tree of $\Bar{G}$, then constraint \eqref{eq:ref gen} is satisfied.

Since set $H\cup\{(a,r_k):k=1,\ldots,m\}$ forms a spanning tree of graph $\bar{\mathcal{G}}$, we have that $H$ is a spanning forest for power system $G$ containing $m$ trees.
If $H$ contains more that $m$ trees, then $H\cup\{(a,r_k):k=1,\ldots,m\}$ cannot form a spanning tree since there must exists some nodes that are disconnected from the spanning tree. 
If $H$ contains less that $m$ trees, then there must exist some reference generators $r_k,r_{k'}\in I_{k''}$, which has been falsified earlier in the proof.
We thus have that $\mathcal{V}_k\cap \mathcal{V}_{k'}=\emptyset$ for any $k\neq k'$.
Furthermore, we have that each tree $I_k$ contained in the spanning forest induced by $H$ must be connected by using the definition of spanning tree.
Combining these arguments yields the result that the islands $I_1,\ldots,I_m$ satisfy constraints \eqref{eq:islanding validity constraint}.
\end{proof}

Lemma \ref{lemma:spanning tree} allows us to encode constraints \eqref{eq:ref gen} and \eqref{eq:islanding validity constraint} using a graphic matroid constraint $E\in\mathcal{B}(\mathcal{M}_{\Bar{G}})$,
where $\mathcal{M}_{\Bar{G}}$ represents the graphic matroid of $\Bar{G}$ and $E$ is the set of edges in the spanning tree. 
We define 
\begin{equation*}
    \bar{\mathcal{D}} = \prod_{k=1}^m\prod_{j\in\mathcal{L}}[0,\Bar{d}_j],
\end{equation*}
and a hybrid matroid $\bar{\mathcal{M}}=(\bar{\mathcal{E}},\bar{\mathcal{D}},\bar{\mathcal{I}})$.
We can thus restrict ourselves to $(E,\Lambda)\in \bar{\mathcal{I}}$ to search for islanding strategies, where $E\in\mathcal{B}(\mathcal{M}_{\bar{\mathcal{G}}})$, $d\in\Lambda$, and $|\Lambda|\leq 1$.
Therefore, we translate optimization program \eqref{eq:islanding formulation} to the following constrained matroid optimization program
\begin{subequations}\label{eq:matroid optimization}
\begin{align}
    \min_{E,d} \quad&\alpha_1F_1(E)+\alpha_2F_2(E) + \alpha_3F_3(E,d)\\
    \st \quad & (E,d)\in\mathcal{B}(\bar{\mathcal{M}})\label{eq:matroid opt constraint}\\
    &\text{Eqn. \eqref{eq:island balance-2}, \eqref{eq:island balance-3}, \eqref{eq:conservation}, \eqref{eq:island capacity}, and \eqref{eq:blackstart}}\label{eq:matroid opt islanding constraint}
\end{align}
\end{subequations}

\subsection{Hybrid Sumodularity-based Algorithm Development }

Solving the constrained matroid optimization in Eqn. \eqref{eq:matroid optimization} is still challenging due to the presence of coupled discrete and continuous variables. 
In what follows, we prove that the objective function and constraint \eqref{eq:matroid opt islanding constraint} are hybrid supermodular in the islanding strategy.
Using the hybrid supermodularity result, we then develop a local search algorithm to efficiently compute the islanding strategy.
We conclude this section by presenting the optimality guarantee of our developed algorithm.

We define an indicator function $\chi_{kj}(E)$ for each $k=1,\ldots,m$ and $j\in\mathcal{V}$ as follows
\begin{equation*}
    \chi_{kj}(E)=\begin{cases}
    1,&\mbox{ if bus $j$ is in island $I_k$,}\\
    0,&\mbox{ otherwise}.
    \end{cases}
\end{equation*}
We have the following preliminary result.
\begin{lemma}\label{lemma:F2 F3 supermodular}
Let $\mathcal{M}_{kj}$ be the graphic matroid of graph $(\mathcal{V},E\cup\{(r_k,j)\})$.
There is a nonincreasing supermodular function $\bar{\chi}_{kj}(E)=n-m-\rho_{\mathcal{M}_{kj}}(E\cup\{(r_k,j)\})$ such that $\chi_{kj}(E)=\bar{\chi}_{kj}(E)$ for any $E\in\mathcal{B}(\mathcal{M}_{\bar{G}})$.
Furthermore, there are nonincreasing supermodular functions $\bar{F}_1(E)$ and $\bar{F}_2(E)$ defined as 
\begin{align*}
    \bar{F}_1(E) &= \sum_{k=1}^m\sum_{j\in\mathcal{G}}(\bar{A}_{kj} + (1-2\bar{A}_{kj})\Omega_{kj}(E)),\\
    \bar{F}_2(E) &= \sum_{k'\neq k}\sum_{(j,j')\in\mathcal{E}}\phi_{k,k',j,j'}(E)|P_{jj'}|,
\end{align*}
such that $\bar{F}_1(E)=F_1(E)$ and $\bar{F}_2(E)=F_2(E)$ for all $E\in\mathcal{B}(\mathcal{M}_{\bar{G}})$, where
\begin{align*}
    \Omega_{kj}(E)&=\begin{cases}
    \bar{\chi}_{kj}(E),&\mbox{ if $\bar{A}_{kj}\leq \frac{1}{2}$},\\
    1-\sum_{k'\neq k}\bar{\chi}_{k'j}(E),&\mbox{ if $\bar{A}_{kj}>\frac{1}{2}$}.
    \end{cases}\\
    \phi_{k,k',j,j'}(E) &= \begin{cases}
    1,&\mbox{ if $j\in I_k$ and $j'\in I_{k'}$},\\
    0,&\mbox{ otherwise}.
    \end{cases}
\end{align*}
\end{lemma}
\begin{proof}
The monotonicty and supermodularity of function $\bar{\chi}$ follow from the monotonicity and submodularity of the rank function of a matroid \cite{oxley2006matroid}.
When $E\in\mathcal{B}(\mathcal{M}_{\bar{G}})$, then we have that $\chi_{kj}(E)=1-(\rho_{\mathcal{M}_{kj}}(E\cup\{(r_k,j)\})-\rho_{\mathcal{M}}(E))$, where $\mathcal{M}$ is the graphic matroid of $G$ \cite{sahabandu2022submodular}. 
Note that $\rho_{\mathcal{M}}(E) = n-m-1$ since the number of edges in a spanning tree is equal to the number of nodes minus one. We thus have that $\bar{\chi}_{kj}(E)=\chi_{kj}(E) = 1-(\rho_{\mathcal{M}_{kj}}(E\cup\{(r_k,j)\})-\rho_{\mathcal{M}}(E))$.

Let $\bar{A}$ be the coherence matrix \cite{chow1982time}. 
We then show that there exists a nonincreasing and supermodular function $\bar{F}_1(E)$
such that $\bar{F}_1(E)=F_1(E)$ for all $E\in\mathcal{B}(\mathcal{M}_{\bar{G}})$.
Since $\bar{\chi}_{kj}(E)$ is supermodular, we have that $\Omega_{kj}(E)$ is also supermodular when $\bar{A}_{kj}\leq \frac{1}{2}$, making $(1-2\bar{A}_{kj})\Omega_{kj}(E)$ further supermodular.
When $\bar{A}_{kj}>\frac{1}{2}$, then $-(1-2\bar{A}_{kj})>0$, making $(1-2\bar{A}_{kj})\Omega_{kj}(E)=(1-2\bar{A}_{kj})-(1-2\bar{A}_{kj})\sum_{k'\neq k}\bar{\chi}_{k'j}(E)$ supermodular.
The statement that $\bar{F}_1(E)=F_1(E)$ can be verified by the definition of Frobenius norm and the fact that $\sum_{k=1}^m\chi_{kj}(E)=1$ for any set $E$ that forms islands.

One can verify that function $\bar{F}_2(E)$ is nonincreasing and supmodular when $E\in\mathcal{B}(\mathcal{M}_{\bar{G}})$ by showing the supermodularity of $\phi_{k,k',j,j'}(E)$, which can be found in \cite{sahabandu2022submodular}.
\end{proof}

Lemma \ref{lemma:F2 F3 supermodular} indicates that metrics $F_1(E)$ and $F_2(E)$ can be rewritten using monotone and supermodular functions $\bar{F}_1(E)$ and $\bar{F}_2(E)$, respectively.
In the following, we focus on metric $F_3(E,d)$ which involves both continuous and discrete variables. 
We define a set of auxiliary variables $d_{kj}$ for each $k=1,\ldots,m$ and $j\in\mathcal{L}$ to model the amount of load attached to bus $j$ when contained in island $I_k$.
By the definition of variable $\chi_{kj}(E)$ and $d_{kj}$, we have that $d_{kj}=0$ must hold when $\chi_{kj}(E)=0$.
As a consequence, we relax metric $F_3(E,d)$ by using a function $\bar{F}_3(E,d)$ as
\begin{equation}
   \bar{F}_3(E,d)=  \min_{d\in\Lambda}\sum_{j\in\mathcal{L}}c_j(\bar{d}_j-\sum_{k=1}^md_{kj}).
\end{equation}
We next show that function $\bar{F}_3(E,d)$ is hybrid monotone and supermodular in the islanding strategy.
\begin{lemma}\label{lemma:F-1 supermodular}
Function $\bar{F}_3(E,d)$ is hybrid monotone nonincreasing and hybrid supermodular in islanding strategy $(E,d)$.
\end{lemma}
\begin{proof}
The lemma holds by the fact that any function of the form $f(\mathcal{S})=\min_{i\in\mathcal{S}}c_i$ is monotone nonincreasing and supermodular.
\end{proof}

In what follows, we focus on the constraints given in Eqn.  \eqref{eq:island balance-2}, \eqref{eq:island balance-3}, \eqref{eq:conservation}, \eqref{eq:island capacity}, and \eqref{eq:blackstart}. 
We relax the constraints by encoding them into  penalty functions, and show that the penalty functions are hybrid monotone and supermodular.

We first consider load-generation balance constraint \eqref{eq:island balance-2}. 
We define a penalty function $F_4(E,d)$ given as below
\begin{equation}\label{eq:penalty-balance}
    F_4(E,d) = \sum_{k=1}^m\left\{\sum_{j\in\mathcal{L}}d_{kj}+\sum_{j\in\mathcal{G}}\sum_{k'\neq k}\chi_{k'j}(E)\bar{g}_j-\sum_{j\in\mathcal{G}}\bar{g}_j\right\}_+
\end{equation}
to penalize the amount of load that exceeds the total power generation within each island, where $\{\cdot\}_+$ denotes $\max\{\cdot,0\}$.

We then consider constraint \eqref{eq:island balance-3}.
We relax constraint \eqref{eq:island balance-3}, by introducing a penalty function 
\begin{equation}\label{eq:penalty-load}
    F_5(E,d) = \sum_{k=1}^m\sum_{j\in\mathcal{L}}\left\{d_{kj} + \sum_{k'\neq k}\chi_{k'j}(E)\bar{d}_j-\bar{d}_j\right\}_+,
\end{equation}
which imposes a positive penalty when the load attached to each load bus $j$ exceeds the limit $\bar{d}_j$.

We next consider constraint \eqref{eq:island capacity}. 
We modify $\bar{G}$ by further introducing a sink node $s$ and a set of edges $\mathcal{A} = \{(j,s):j\in\mathcal{L}\}$.
Edge set $\mathcal{A}$ connects each load bus $j\in\mathcal{L}$ with the sink node $s$.
We further define the capacities for the edges in $\mathcal{A}$ to be infinity, and thus no capacity constraint can be violated by any $(j,s)\in\mathcal{A}$.
We then have the following result.
\begin{proposition}\label{prop:redirect overflow}
Consider the augmented graph $\Bar{G}$ with sink node $s$ and edge set $\mathcal{A}$. Given the load $d_j\leq\Bar{d}_j$ for all $j\in\mathcal{L}$, we have that there exists some post-islanding power flow satisfying $\Tilde{P}_{jj'}\leq \Bar{P}_{jj'}$ for all $(j,j')\in\Bar{\mathcal{E}}$ and conservation law given in Eqn. \eqref{eq:conservation}. 
\end{proposition}
\begin{proof}
Suppose that there exists some bus $j\in\mathcal{V}$ such that constraints \eqref{eq:conservation} and \eqref{eq:island capacity} cannot be satisfied simultaneously. 
We denote the post-islanding power flow that creates capacity constraint violation as $\hat{P}_{jj'}$.
Without loss of generality, we consider that bus $j$ is inlcuded in island $I_k$.
We let $j''\in\mathcal{V}_k\cap\mathcal{L}$ be some load bus that is contained in island $I_k$. 
Since each island is connected, there must exist some path $\eta$ from bus $j$ to $j''$.
With a slight abuse of notation, we use $(t,j')\in \eta$ to represent that a transmission line $(t,j')$ is on path $\eta$.

By introducing the sink node $s$, we can ensure that constraints \eqref{eq:conservation} and \eqref{eq:island capacity} are satisfied by first choosing power flow $\tilde{P}_{jj'}$ satisfying  \eqref{eq:island capacity} for all $j'$ along the path from $j$ to $j''$, and then letting $\tilde{P}_{j''s}=\sum_{(t,j')\in\eta}(\hat{P}_{tj'}-\tilde{P}_{tj'})$ such that Eqn. \eqref{eq:conservation} is met.
Such power flow $\tilde{P}_{j''s}$ always exists since $\tilde{P}_{j''s}\in\mathbb{R}$.
\end{proof}

Proposition \ref{prop:redirect overflow} allows us to verify the satisfaction of constraint \eqref{eq:island capacity} by verifying whether $\sum_{j\in\mathcal{L}}|P_{js}|=0$ holds or not.
We observe that the sink node and edge set $\mathcal{A}$ do not exist in the power system $G$, and hence we can view all auxiliary transmission line being tripped by the islanding strategy.
Therefore, we have that the transmission line capacity constraint can be penalized by the total amount of power flow disruption over the edges in $\mathcal{A}$.
We modify function $\bar{F}_2(E)$ as 
\begin{equation}\label{eq:penalty-capacity}
    \bar{F}_2(E) = \sum_{k'\neq k}\sum_{(j,j')\in\mathcal{E}}\phi_{k,k',j,j'}(E)|P_{jj'}| + \beta\sum_{j\in\mathcal{L}}|\tilde{P}_{js}|
\end{equation}
to jointly penalize (i) the power flow disruption due to tripping transmission lines, and (ii) any violation of the capacity constraints in Eqn. \eqref{eq:island capacity},
where $\beta>0$ is a constant modeling the trade-off between (i) and (ii).

Consider constraint \eqref{eq:blackstart} for the blackstart generator allocation. 
We define a penalty function
\begin{multline}\label{eq:penalty-blackstart}
    F_6(E) = \Big\{\sum_{k'\neq k}\sum_{j\in \mathcal{V}_{k'}}\chi_{k'j}(E)\mathbb{I}_j+\sum_{k'\neq k}\sum_{j\in \mathcal{V}_{k'}}\chi_{k'j}(E)\\
    -\sum_{j\in\mathcal{V}}\mathbb{I}_j+1\Big\}_+,
\end{multline}
to penalize the scenarios where some island contains no blackstart generator.

Using the penalty functions given in Eqn. \eqref{eq:penalty-balance} to \eqref{eq:penalty-blackstart}, we reformulate the constrained matroid optimization program in Eqn. \eqref{eq:matroid optimization} as the following unconstrained program
\begin{subequations}\label{eq:unconstrained matroid optimization}
\begin{align}
    \min_{E,d} \quad&\alpha_1\bar{F}_1(E)+\alpha_2\bar{F}_2(E) + \alpha_3\bar{F}_3(E,d) + \alpha_4F_4(E,d)\nonumber\\
    &\quad\quad+\alpha_5F_5(E,d) + \alpha_6F_6(E)\label{eq:unconstrained matroid optimization obj}\\
    \st \quad & (E,d)\in\mathcal{B}(\bar{\mathcal{M}})\label{eq:unconstrained matroid opt constraint}
\end{align}
\end{subequations}
where parameters $\alpha_1$ to $\alpha_6$ are positive constants modeling the trade-off among metrics $\bar{F}_1(E)$, $\bar{F}_2(E)$, $\bar{F}_3(E,d)$ as well the constraints in Eqn. \eqref{eq:island balance-2}, \eqref{eq:island balance-3}, \eqref{eq:island capacity}, and \eqref{eq:blackstart}.

In what follows, we convert the matroid optimization program in Eqn. \eqref{eq:unconstrained matroid optimization} to an equivalent hybrid submodular optimization problem.
We define 
\begin{align*}
    \bar{F}_4(E,d) &= \sum_{k=1}^m\Big\{\min_{d\in\Lambda}\sum_{j\in\mathcal{L}}d_{kj}
    +\sum_{j\in\mathcal{G}}\sum_{k'\neq k}\chi_{k'j}(E)\bar{g}_j-\sum_{j\in\mathcal{G}}\bar{g}_j\Big\}_+\\
    \bar{F}_5(E,d) &=\sum_{k=1}^m\sum_{j\in\mathcal{L}}\Big\{\min_{d\in\Lambda}d_{kj}
    +\sum_{k'\neq k}\chi_{k'j}(E)\bar{d}_j-\bar{d}_j\Big\}_+\\
    \bar{F}_6(E) &=\sum_{k=1}^m\Big\{\sum_{k'\neq k}\sum_{j\in \mathcal{V}_{k'}}\chi_{k'j}(E)\mathbb{I}_j\\
    &\quad\quad\quad\quad+\sum_{k'\neq k}\sum_{j\in \mathcal{V}_{k'}}\chi_{k'j}(E)
    -\sum_{j\in\mathcal{V}}\mathbb{I}_j+1\Big\}_+\\
    \bar{F}(E,\Lambda) &= \alpha_1\bar{F}_1(E)+\alpha_2\bar{F}_2(E)+\alpha_3\bar{F}_3(E,d)\\
    &\quad\quad\quad\quad+\alpha_4\bar{F}_4(E,d) +\alpha_5\bar{F}_6(E,d)+\alpha_6\bar{F}_6(E).
\end{align*}
Given the definition of function $\bar{F}(E,\Lambda)$, we reformulate the optimization problem in Eqn. \eqref{eq:unconstrained matroid optimization} as 
\begin{subequations}\label{eq:hybrid submodular matroid opt}
\begin{align}
    \min \quad&\bar{F}(E,\Lambda)\\
    \st \quad&(E,\Lambda)\in\mathcal{B}(\bar{\mathcal{M}})
\end{align}
\end{subequations}
The equivalence between optimization programs \eqref{eq:hybrid submodular matroid opt} and \eqref{eq:unconstrained matroid optimization} is established as follows.
\begin{proposition}
An islanding strategy $(E,d)$ is an optimal solution to Eqn. \eqref{eq:unconstrained matroid optimization} if and only if $(E,\Lambda)$ is an optimal solution to Eqn. \eqref{eq:hybrid submodular matroid opt}, where $\Lambda = \{d\}$.
Furthermore, $\bar{F}(E,\Lambda) = F(E,d)$.
\end{proposition}
\begin{proof}
The equivalence between $(E,d)$ and $(E,\Lambda)$ is established by $\Lambda = \{d\}$.
One can verify that $\bar{F}(E,\Lambda)=F(E,d)$ when $\Lambda=\{d\}$ by using the one to one correspondence between each term in $\bar{F}(E,\Lambda)$ and Eqn. \eqref{eq:unconstrained matroid optimization obj}.
\end{proof}

We further establish the following monotonicity and supermodularity properties for $\bar{F}(E,\Lambda)$ as given below.
\begin{theorem}\label{thm:supermodularity}
The function $\bar{F}(E,\Lambda)$ is hybrid monotone nonincreasing and supermodular in $(E,\Lambda)$.
\end{theorem}
\begin{proof}
We prove the theorem by showing that each term in $\bar{F}(E,\Lambda)$ is hybrid monotone nonincreasing and supermodular.

The monotonicity and supermodularity of functions $\bar{F}_1(E)$ and $\bar{F}_2(E)$ hold by Lemma \ref{lemma:F2 F3 supermodular}.
The monotonicity and supermodularity of $\alpha_3\bar{F}_3(E,d)$ follows from Lemma \ref{lemma:F-1 supermodular}.

Using Lemma \ref{lemma:F2 F3 supermodular} and Definition \ref{def:hybrid submodular}, we have that $\min_{d\in\Lambda}\sum_{j\in\mathcal{L}}d_{kj}+\sum_{j\in\mathcal{G}}\sum_{k'\neq k}\chi_{k'j}(E)\bar{g}_j-\sum_{j\in\mathcal{G}}\bar{g}_j$ is hybrid monotone nonincreasing and supermodular. 
In addition, we have that function $\max\{f(S),c\}$ is monotone nonincreasing and supermodular for any constant $c$ if $f(S)$ is monotone nonincreasing and supermodular. 
We thus have that 
\begin{equation*}
    \alpha_4\sum_{k=1}^m\Big\{\min_{d\in\Lambda}\sum_{j\in\mathcal{L}}d_{kj}+\sum_{j\in\mathcal{G}}\sum_{k'\neq k}\chi_{k'j}(E)\bar{g}_j-\sum_{j\in\mathcal{G}}\bar{g}_j\Big\}_+
\end{equation*}
is monotone nonincreasing and supermodular. 
Similar arguments can be used to show that the term 
\begin{equation*}
    \alpha_5\sum_{k=1}^m\sum_{j\in\mathcal{L}}\Big\{\min_{d\in\Lambda}d_{kj}+\sum_{k'\neq k}\chi_{k'j}(E)\bar{d}_j-\bar{d}_j\Big\}_+
\end{equation*}
is hybrid monotone nonincreasing and supermodular. 


The hybrid monotonicity and supermodularity of the term 
\begin{multline*}
    \alpha_6\sum_{k=1}^m\Big\{\sum_{k'\neq k}\sum_{j\in \mathcal{V}_{k'}}\chi_{k'j}(E)\mathbb{I}_j+\sum_{k'\neq k}\sum_{j\in \mathcal{V}_{k'}}\chi_{k'j}(E)
    \\
    -\sum_{j\in\mathcal{V}}\mathbb{I}_j+1\Big\}_+
\end{multline*}
follows by Lemma \ref{lemma:F2 F3 supermodular}.

Applying the argument that linear combinations of monotone and supermodular functions with positive weights are still monotone and supermodular completes the proof.
\end{proof}

  \begin{center}
  	\begin{algorithm}[!htp]
  		\caption{Local search algorithm for controlled islanding}
  		\label{algo:islanding}
  		\begin{algorithmic}[1]

		\State Set $\epsilon \in [0, 0.5)$
		\State Initialize $(E,d)$ to define a valid islanding strategy\label{line:init}
		\State $found \leftarrow 1$
		\While{$found ==1$}
		\State $found \leftarrow 0$
       \For{$(j_{1},j_{1}^{\prime}) \in \mathcal{E}\setminus E, \ (j_{2},j_{2}^{\prime}) \in E$}
       \If{$E \cup \{(j_{1},j_{1}^{\prime})\} \setminus \{(j_{2},j_{2}^{\prime})\} \in \mathcal{B}(\mathcal{M}_{0})$}\label{line:update}
		\State Compute $d'$ and $\tilde{P}_{jj'}$ using Eqn. \eqref{eq:post-islanding power flow} 
       \If{$\overline{F}(E \cup \{(j_{1},j_{1}^{\prime})\} \setminus \{(j_{2},j_{2}^{\prime})\},\{d'\}) < (1-\epsilon)\bar{F}(E,\{d\})$}\label{line:eval start}
        \State $E \leftarrow E \cup \{(j_{1},j_{1}^{\prime})\} \setminus \{(j_{2},j_{2}^{\prime})\}$, $d \leftarrow d^{\prime}$
        \State $found \leftarrow 1$
        \State \textbf{Break}\label{line:break}
        \EndIf\label{line:eval end}
        \EndIf
        \EndFor
        \EndWhile
        \State \Return $E,d$
  		\end{algorithmic}
  	\end{algorithm}
  \end{center}

The monotonicity and supermodularity properties established by Theorem \ref{thm:supermodularity} allow us to develop an efficient local search solution algorithm, as shown in Algorithm \ref{algo:islanding}, with provable optimality guarantee.
The algorithm first initializes an islanding strategy $(E,d)$ in line \ref{line:init}.
Note that the initial set $E$ should not include any transmission line from the auxiliary transmission line set $\mathcal{A}$.
The algorithm then proceeds in an iterative manner.
At each iteration, the algorithm generates a new islanding strategy by including an unselected transmission line $(j_1,j_1')\in\mathcal{E}\setminus E$ and excluding a selected transmission line $(j_2,j_2')$, as shown in line \ref{line:update}.
Note that transmission line $(j_1,j_1')$ will connect two disjoint islands, and $(j_2,j_2')$ will be used to ensure the $m$ islands are generated with all reference generators being not connected with each other.
Then the algorithm computes the continuous variable $d'$ and $P$ by solving the following linear program
\begin{subequations}\label{eq:post-islanding power flow}
\begin{align}
    \min_{d',\tilde{P}}\quad&F_3(E,d')+\gamma\sum_{j\in\mathcal{V}}|P_{js}|\\
    \st\quad&\text{Eqn. \eqref{eq:island balance-2}, \eqref{eq:island balance-3}, \eqref{eq:conservation} and \eqref{eq:island capacity}}
\end{align}
\end{subequations}
where $\gamma>0$ captures the weight assigned to load shedding cost and violations of capacity constraints.
Line \ref{line:eval start} to line \ref{line:eval end} evaluates the performance of the islanding strategy $(E,d')$ using function $\bar{F}(E,\Lambda)$ with $\Lambda=\{d'\}$. 
The algorithm finally decides whether the islanding strategy should be updated (if the condition in line \ref{line:eval start} holds) or not (line \ref{line:break}).

We conclude this section by presenting the optimality guarantee provided by Algorithm \ref{algo:islanding}.
\begin{theorem}
Let $Z$ be a sufficiently large positive number such that $Z-\bar{F}(E,\Lambda)\geq 0$ for all $(E,\Lambda)$ where $|\Lambda|\leq 1$.
As parameter $\epsilon\rightarrow 0$, Algorithm \ref{algo:islanding} returns an islanding strategy $(E,d)$ such that 
\begin{equation*}
    Z-\bar{F}(E,\Lambda)\geq \frac{1}{2}(Z-\bar{F}(E',\Lambda'))
\end{equation*}
for all $(E',\Lambda')\in\mathcal{B}(\bar{\mathcal{M}})$, where $\Lambda = \{d\}$ and $|\Lambda'|\leq 1$.
\end{theorem}
\begin{proof}
By Theorem \ref{thm:supermodularity}, we have that $\bar{F}(E,\Lambda)$ is hybrid supermodular.
Therefore, $Z-\bar{F}(E,\Lambda)$ is hybrid submodular.
By applying Proposition \ref{prop:optimality} yields the theorem.
\end{proof}

Denote $E_{t}$ and $d_t$ as the set of edges and load shedding obtained after $t$ iterations by using Algorithm \ref{algo:islanding}. 
We have $\overline{F}(E_{t},d_t) < (1-\epsilon)^{t}\overline{F}(E_{0},d_0)$. 
Hence we have that the algorithm terminates within $\lceil \frac{\log{\left\{\frac{\overline{F}(E_{t})}{\overline{F}(E_{0})}\right\}}}{\log{(1-\epsilon)}}\rceil$ iterations and each iteration has worst case $O(|\mathcal{E}|^{2}\xi)$ complexity, yielding an upper bound of $$O\left(\lceil \frac{\log{\left\{\frac{\overline{F}(E_{t})}{\overline{F}(E_{0})}\right\}}}{\log{(1-\epsilon)}}|\mathcal{E}|^{2}\xi\rceil\right)$$ on the complexity, where $\xi$ is the computational complexity of solving the linear program in Eqn. (18).

\section{Numerical Study}\label{sec:experiment}

This section presents three case studies to compare our proposed approach and the state-of-the-art MILP-based solution.
We first present necessary background on the MILP formulation. 
We then present a comparison of results from the MILP baseline and our approach on the IEEE 118-bus system \cite{118bus}, IEEE 300-bus system \cite{300bus}, ActivSg 500-bus system \cite{birchfield2016grid}, and Polish 2383-bus system \cite{zimmerman2010matpower}.
We consider that at time $t=0$, buses 10, 69, 18, and 10 incur three-phase faults for IEEE 118-bus, 300-bus, ActivSg 500-bus, and Polish 2383 test cases, respectively, which necessitates controlled islanding.
The fault is cleared at time 0.4s.

\subsection{Baseline: Mixed Integer Linear Program}
This subsection introduces the baseline approach that utilizes mixed-integer linear program (MILP) to solve the islanding problem. 
The baseline approach is developed based on \cite{kyriacou2017controlled,patsakis2019strong}.

The baseline approach takes the coherent generator groups, denoted as $\mathcal{C}_1,\ldots,\mathcal{C}_m$, as input, and computes the load shedding at each load bus, the transmission lines to trip to partition the power system, and the post-islanding power flow.
We summarize the notations and variables used in the MILP in Table \ref{tab:variable}.
The MILP given as below minimizes the costs incurred by load shedding and power flow disruption.
\begin{subequations}\label{eq:MILP}
\begin{align}
    \min\quad &\alpha_2F_2(I_1,\ldots,I_m)+\alpha_3F_3(I_1,\ldots,I_m)\label{eq:MILP obj}\\
    \mbox{s.t.}\quad&\sum_{j\in \mathcal{L}}d_{kj}\leq \sum_{j\in\mathcal{G}}\Bar{g}_jx_{kj},~\forall k=1,\ldots,m,~\forall j\in \mathcal{V}\label{eq:MILP c1} \\
    &0\leq d_{kj}\leq \Bar{d}_jx_{kj},~\forall k=1,\ldots,m,~\forall j\in \mathcal{V}\label{eq:MILP c2}\\
    &x_{kj} = v_{kj'},~\forall j,j'\in\mathcal{C}_k\label{eq:MILP c3}\\
    & w_{k,j,j'}\in\{0,1\},~\forall k=1,\ldots,m,\forall (j,j')\in\mathcal{E}\label{eq:MILP c4}\\
    & x_{kj}\in\{0,1\},~\forall k=1,\ldots,m,\forall j\in\mathcal{V}\label{eq:MILP c5}\\
    & z_{jj'}\in\{0,1\},~\forall (j,j')\in\mathcal{E}\label{eq:MILP c6}\\
    &\sum_{k=1}^mx_{kj}\leq 1,~\forall j\in \mathcal{V}\label{eq:MILP c7}\\
    & w_{k,j,j'}\leq x_{kj}, ~w_{k,j,j'}\leq x_{kj'},~\forall k, (j,j')\in\mathcal{E}\label{eq:MILP c8}\\
    & z_{jj'} = \sum_{k=1}^m w_{k,j,j'}, z_{jj'}=z_{j'j},~\forall (j,j')\in\mathcal{E}\label{eq:MILP c9}\\
    & 0\leq l_{k,j,j'}\leq Zz_{jj'},~\forall (j,j')\in \mathcal{E}\label{eq:MILP c11}\\
    &v_{kj}\sum_{j\in\mathcal{V}}x_{kj}-x_{kj}+\sum_{j\in\mathcal{N}(j')}f_{k,j,j'} = \sum_{j\in\mathcal{N}(j')}f_{k,j',j},\nonumber\\
    &\quad\quad\quad\quad\quad\quad\forall j\in\mathcal{V},k=1,\ldots,m\label{eq:MILP c12}\\
    & \Tilde{P}_{jj'}\leq \Bar{P}_{jj'}z_{jj'},~\forall (j,j')\in\mathcal{E}\label{eq:MILP c13}\\
    &\sum_{j\in\mathcal{N}(j')}\Tilde{P}_{jj'}+g_{j'}-d_{j'}=0,~\forall j'\in\mathcal{V}\label{eq:MILP c14}\\
    &\sum_jx_{kj}\mathbb{I}_j\geq 1,~\forall k=1,\ldots,m\label{eq:MILP c15}
\end{align}
\end{subequations}
In order to retain linearity, we note from Eqn. \eqref{eq:MILP obj} that the MILP cannot optimize over generator coherency in Eqn. \eqref{eq:coherence} as our proposed approach does.
In the MILP, function $F_3(I_1,\ldots,I_m)$ can be represented as $F_3(I_1,\ldots,I_m) = \sum_{(j,j')\in\mathcal{E}}(1-z_{jj'})\frac{|P_{jj'}| + |P_{j'j}|}{2}$.
Constraint \eqref{eq:MILP c1} and \eqref{eq:MILP c2} require the admissible loads at the each load bus to satisfy the load-generation balance.
Constraint \eqref{eq:MILP c3} specifies the reference generator of each island $I_k$.
Constraints \eqref{eq:MILP c4} to \eqref{eq:MILP c9} ensure that variables $x$, $w$, and $z$ define a set of disjoint islands.
Constrains \eqref{eq:MILP c11} and \eqref{eq:MILP c12} define an auxiliary flow $l_{k,j,j'}$ initiating from the reference generator to ensure the flow conservation law and connectivity within each island.
Constraints \eqref{eq:MILP c13} and \eqref{eq:MILP c14} define the post-islanding power flow $\Tilde{P}_{jj'}$ on each transmission line.
Constraint \eqref{eq:MILP c15} assigns the blackstart generators to ensure efficient restoration.
The MILP formulated in Eqn. \eqref{eq:MILP} involves $3m|\mathcal{V}|+2m|\mathcal{E}|+2|\mathcal{E}|$ number of decision variables and $(4m+2)|\mathcal{V}|+(3m+7)|\mathcal{E}|+m$ number of constraints.

\begin{table*}[htb]
\centering
\caption{This table contains the set of variables used in the MILP formulation. The first column gives the notation of each variable. The last column presents the interpretation of each variable.}
\begin{tabular}{|c|c|c|} 
 \hline
 \textbf{Notation} & \textbf{Type} & \textbf{Interpretation} \\ [0.5ex] 
 \hline
 $x_{kj}$ & Binary & Whether bus $j$ is contained in island $I_k$ ($x_{kj}=1$) or not ($x_{kj}=0$) \\
 \hline
 $w_{k,j,j'}$ & Binary & Whether transmission line $(j,j')\in\mathcal{E}$ is contained in island $I_k$ ($z_{k,j,j'}=1$) or not ($z_{k,j,j'}=0$) \\
 \hline
 $z_{jj'}$ & Binary & Whether transmission line $(j,j')$ will be tripped ($z_{jj'}=0$) or not ($z_{jj'}=1$) \\
 \hline
 $v_{kj}$ & Binary & Indicator for the reference generator $j$ in island $I_k$\\
 \hline
 $l_{k,j,j'}$ & Nonnegative real & Auxiliary flow on transmission line $(j,j')$ when included in island $I_k$ \\  
 \hline
 $Z$ & Positive real & A sufficiently large positive number\\
 \hline
 $\mathcal{N}(j)$ & Set & The neighboring buses $\{j':(j',j)\in\mathcal{E}\}$ of $j$\\
 \hline
\end{tabular}
\label{tab:variable}
\end{table*}

We solve the MILP formulated in \eqref{eq:MILP} by using two solvers: intlinprog provided by Matlab \cite{intlinprog} and mixed-integer program solver provided by Gurobi \cite{gurobi}. We refer to the former as MILP-M, while the latter as MILP-G in the rest of this section.
Since MILP-G and MILP-M cannot optimize the generator coherency, we compute the generator coherency for MILP-M and MILP-G beforehand for the purpose of comparison.

\begin{figure*}[!ht]
  \centering
    \subfloat[IEEE 300-bus, $m=3$ using MILP.]{\includegraphics[width=0.495\textwidth]{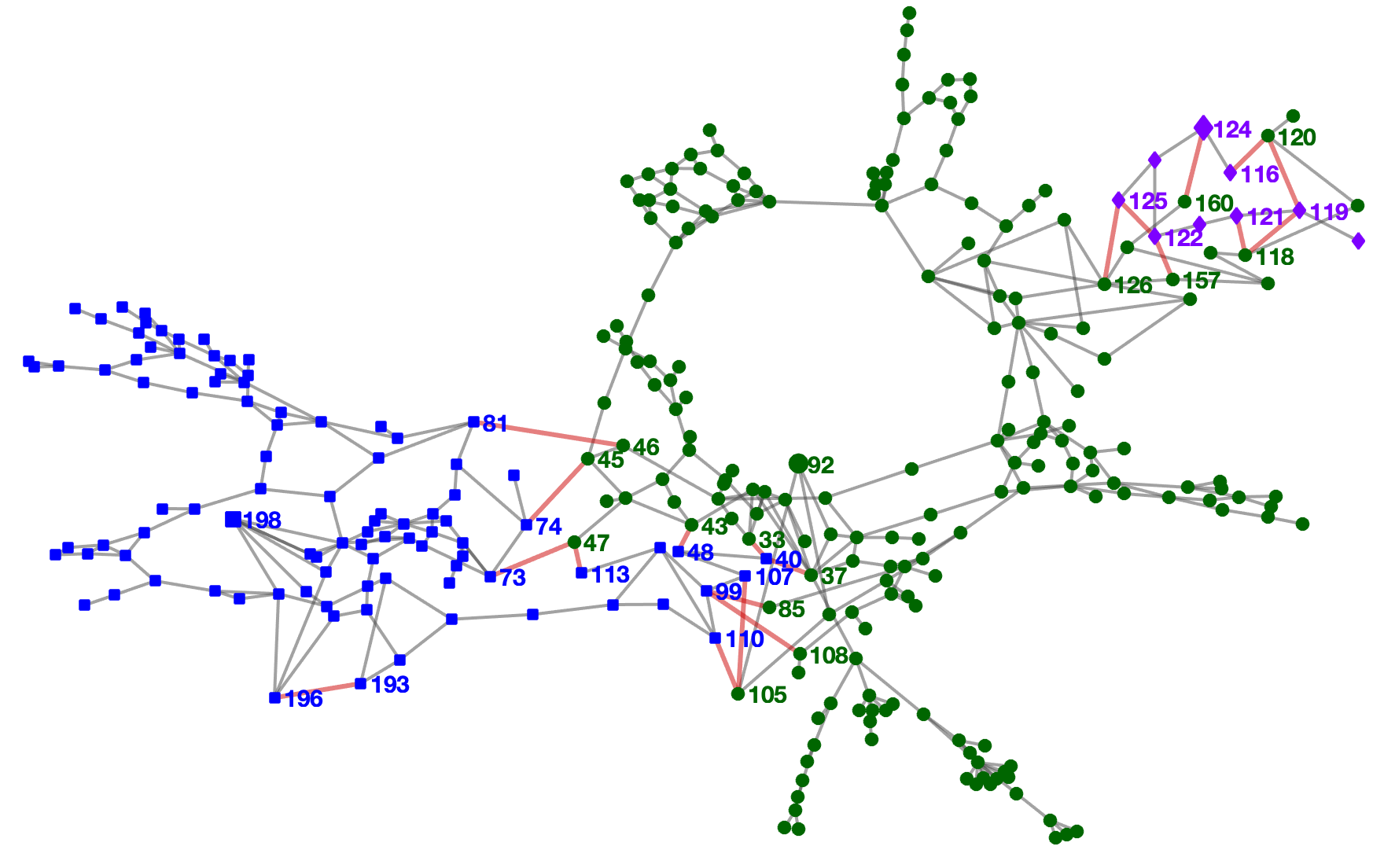}\label{fig:f3}}
  \hfill
  \subfloat[IEEE 300-bus, $m=3$ using Algorithm~1.]{\includegraphics[width=0.495\textwidth]{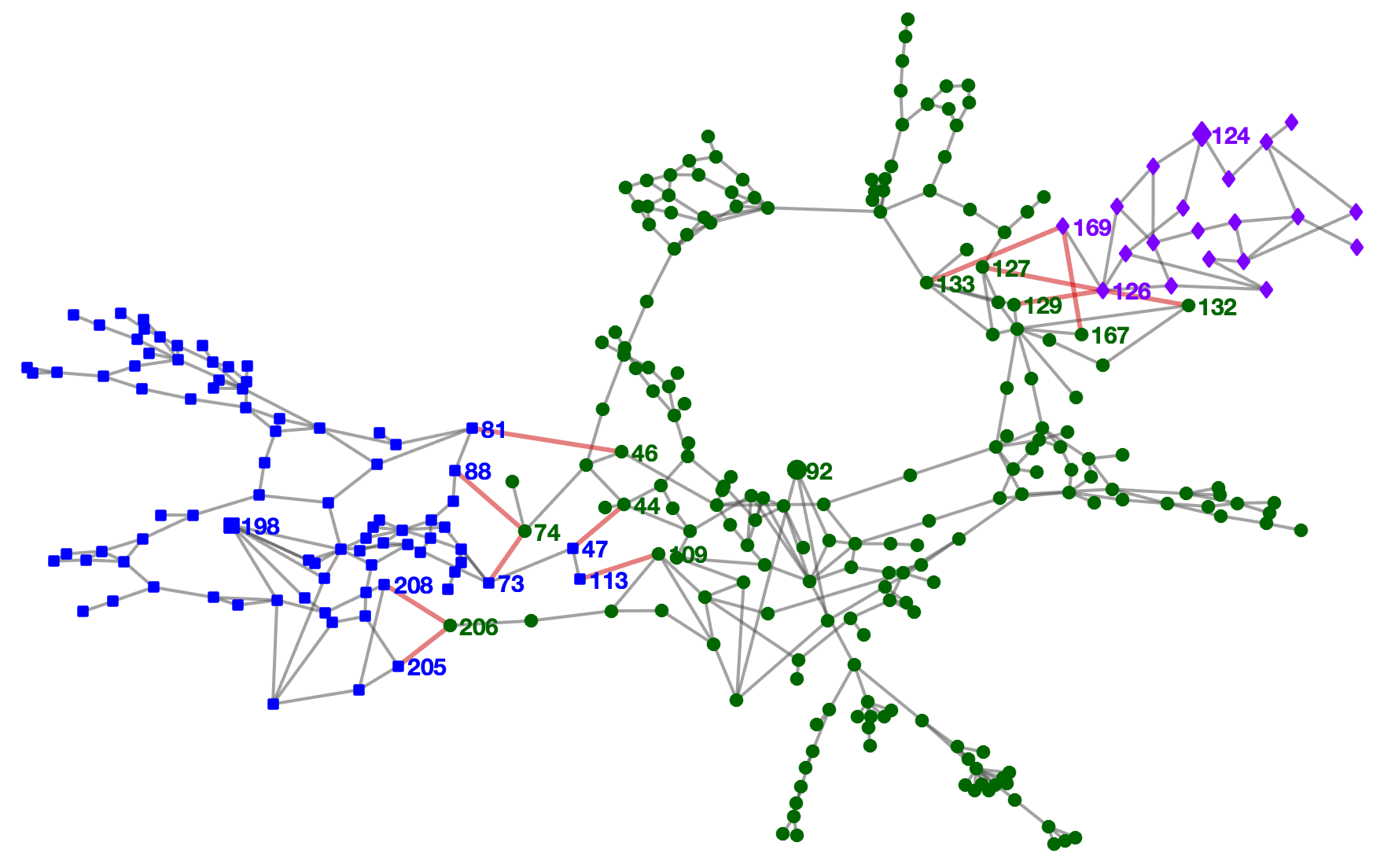}\label{fig:f4}}
  \hfill
     \subfloat[ActivSg 500-bus, $m=3$ using MILP.]{\includegraphics[width=0.495\textwidth]{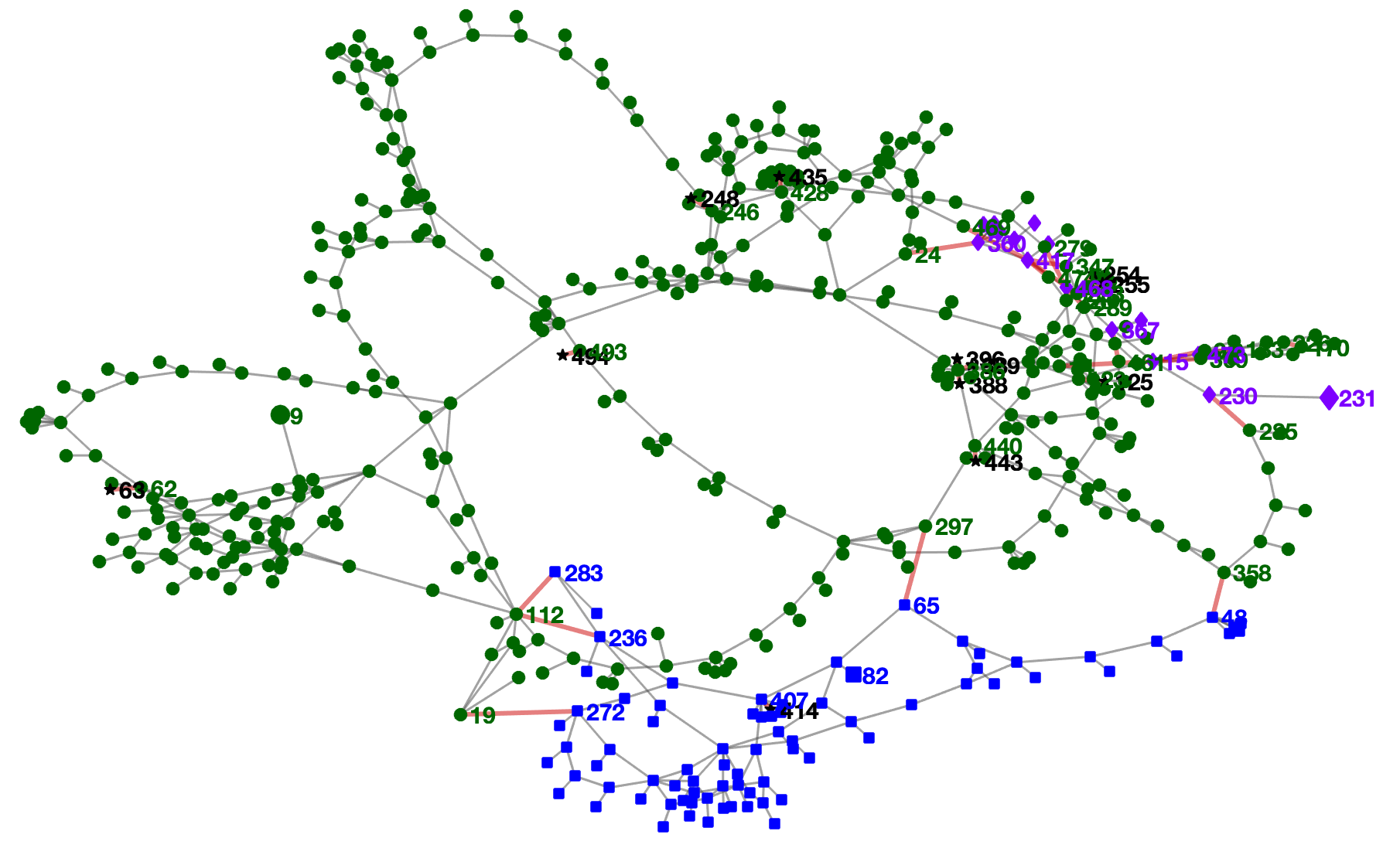}\label{fig:f5}}
  \hfill
  \subfloat[ActivSg 500-bus, $m=3$ using Algorithm~1.]{\includegraphics[width=0.495\textwidth]{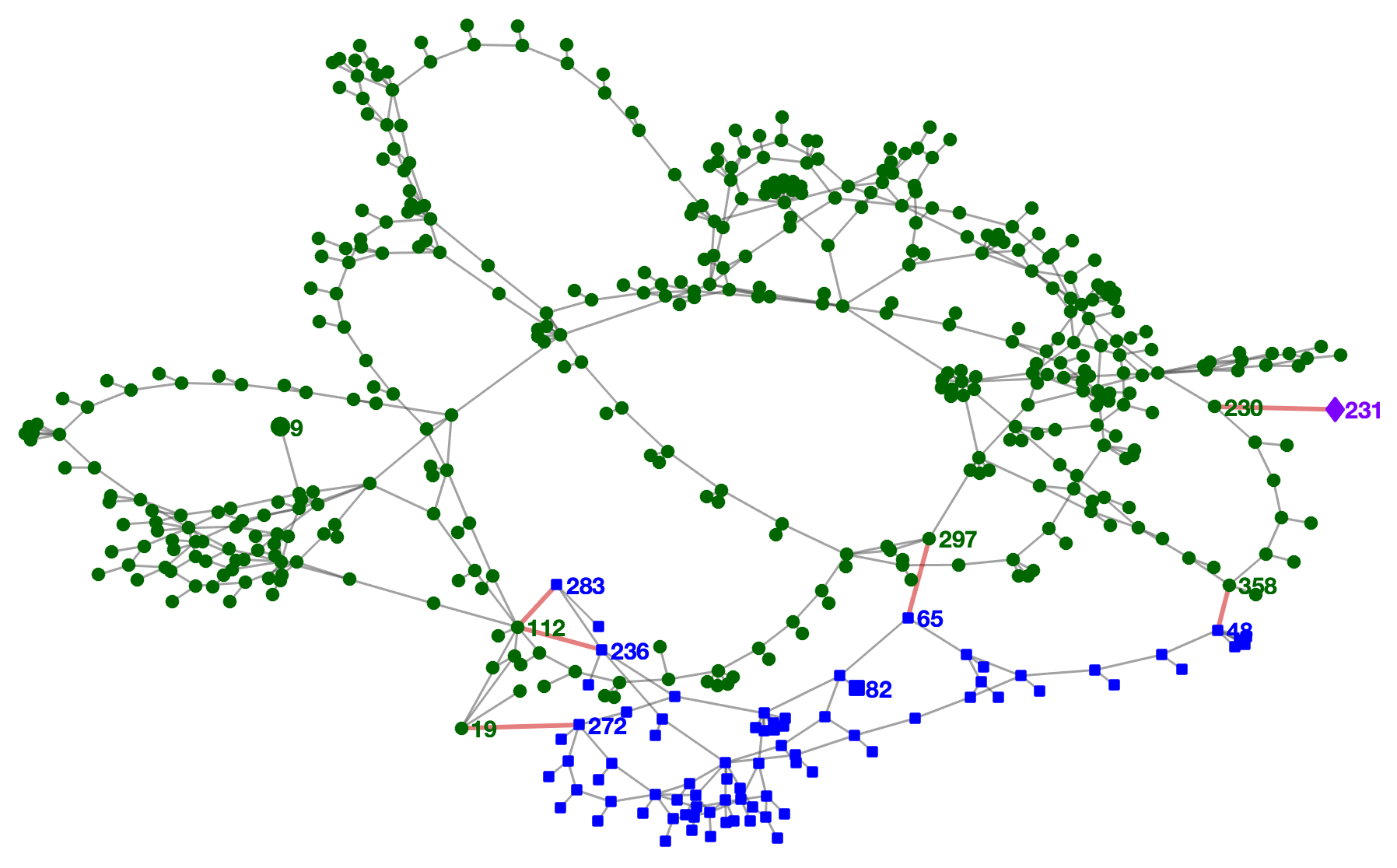}\label{fig:f6}}
  \caption{Comparison of islanding strategies given by MILP-based approach   and Algorithm~1 for IEEE 300-bus and ActivSg 500-bus test cases with $m=3$ islands (in blue, green, and purple colors). Each node in the graph represents a bus of power system. The edges correspond to the set of transmission lines. The set of transmission lines tripped in each case is marked in red, whereas the those remain in the power system are in grey color. Isolated buses are marked in black color.}\label{fig:islands}
\end{figure*}

\subsection{Simulation Results}

In our case study, we choose the load shedding cost function $c_j(x)$ as $c_j(x)=\zeta_jx_j$, where $\zeta_j$ is randomly generated following a uniform distribution within range $(1,100)$.
Trade-off parameters $\alpha_1,\ldots,\alpha_6$ are chosen to be $1$.
We choose parameter $\epsilon=0$.
We compare our proposed approach in Algorithm \ref{algo:islanding} with the baseline approach in two settings for the evaluation purpose, where the first setting partitions the test cases into $m=2$ islands, and the second setting generates $m=3$ islands. 

\begin{table*}[ht]
\caption{This table presents the set of reference generators given to Algorithm~1 and the number of transmission lines tripped by islanding strategies obtained using MILP-G (MILP using Gurobi), MILP-M (MILP using intlinprog), and Algorithm~1 (hybrid submodular approach). The islanding strategy given by Algorithm~1 trips fewer transmission lines in all test cases including IEEE 118-bus, IEEE 300-bus, ActivSg 500-bus, and Polish 2383-bus systems with different settings, and thus the proposed approach is more practical for implementation.}\label{table:edges-removed}
\centering
\scalebox{1}{
\begin{tabular}{|c|c|c|ccc|}
\hline
\multirow{2}{*}{\textbf{Test Case}}      & \multirow{2}{*}{\textbf{\begin{tabular}[c]{@{}c@{}}Number \\ of Islands\end{tabular}}} & \multicolumn{1}{c|}{\multirow{2}{*}{\textbf{\begin{tabular}[c]{@{}c@{}}Ref. Generators\\ (for Algorithm 1)\end{tabular}}}} 
& \multicolumn{3}{c|}{\textbf{Number of Transmission Lines Tripped}}                       \\ \cline{4-6} 
                                          &         &                                                                            & \multicolumn{1}{c|}
                                          {~~~~~~~~\textbf{MILP-G}~~~~~~~~} &
                                          \multicolumn{1}{c|}{~~~~~~~~\textbf{MILP-M}~~~~~~~~} & \textbf{Algorithm 1} \\ \hline
\multirow{2}{*}{\textbf{IEEE 118-bus}}    & 2       & 12, 100                                                                               & \multicolumn{1}{c|} {9} & \multicolumn{1}{c|} {9}    & {$\mathbf{5}$}             \\ \cline{2-6} 
                                          & 3        & 26, 65, 80                                                                              & \multicolumn{1}{c|}{14} & \multicolumn{1}{c|}{14}  & {$\mathbf{8}$}          \\ \hline
\multirow{2}{*}{\textbf{IEEE 300-bus}}    & 2      & 91, 198                                                                                &\multicolumn{1}{c|}{14} & \multicolumn{1}{c|}{5}   & {$\mathbf{4}$ }        \\ \cline{2-6} 
                                          & 3      & 92, 124, 198                                                                                & \multicolumn{1}{c|}{18} & \multicolumn{1}{c|}{20}   & {$\mathbf{12}$}             \\ \hline
\multirow{2}{*}{\textbf{ActivSg 500-bus}}    & 2      & 9, 16                                                                                & \multicolumn{1}{c|}{34} & \multicolumn{1}{c|}{42}   & $\mathbf{11}$               \\ \cline{2-6} 
                                          & 3       & 9, 82, 231                                                                               & \multicolumn{1}{c|}{45} & \multicolumn{1}{c|}{32}   & {$\mathbf{6}$}              \\ \hline
\multirow{2}{*}{\textbf{Polish 2383-bus}} & 2      & 41, 1726                                                                                & \multicolumn{1}{c|}{NA} & \multicolumn{1}{c|}{NA}            & $\mathbf{3}$              \\ \cline{2-6} 
                                          & 3      & 45, 125, 1106                                                                          & \multicolumn{1}{c|}{NA} &\multicolumn{1}{c|}{NA}            &          $\mathbf{11}$                \\ \hline
\end{tabular}
}
\end{table*}

Table~\ref{table:edges-removed} presents the numbers of transmission lines removed by the baseline and our approach.
We observe that the islanding strategies given by Algorithm \ref{algo:islanding} trip fewer transmission lines compared to the baseline in all test cases, and hence are more practical to be implemented.
Fig. \ref{fig:islands} compares the islanding strategies given by MILP and Algorithm \ref{algo:islanding} for IEEE 300-bus and ActivSg 500-bus test cases when the desired number of islands is set to $m=3$.

We summarize the amount of load shedding, the values of objective function $\Bar{F}(E,\Lambda)$, and run time of MILP-M, MILP-G, and Algorithm \ref{algo:islanding} in Table \ref{table:results2island}.
We note that for Polish 2383-bus system, MILP-M and MILP-G run out of memory (``NA" entries in Table \ref{table:results2island}), and thus do not give an islanding solution.
We observe that our approach outperforms MILP-M and MILP-G in terms of the total cost $\Bar{F}(E,\Lambda)$ for all test cases.
The amount of load shedding required by Algorithm \ref{algo:islanding} is no larger than MILP-M or MILP-G for all test cases.
From the last column of Table \ref{table:results2island}, we note that the run time of the MILP-M grows fast as the scales of test cases increase, making the approach not applicable for Polish 2383-bus system. 
Although MILP-G takes the least run time for small-scale (IEEE 118-bus and IEEE 300-bus) and medium-scale (ActivSg 500-bus) test cases, it is not applicable to large-scale test case (Polish 2383-bus).
Our proposed approach computes a controlled islanding solution for large-scale test cases within a reasonable amount of time. 
We further remark that neither MILP-M nor MILP-G could optimize over generator coherency and thus involve fewer decision variables compared with our approach.

For the ActivSg 500-bus test case with $m=2$ islands, the run time of MILP-G and MILP-M are about $12$ times less than our approach. However, in this scenario the islanding strategies given by MILP-G and MILP-M trip about three to four times as many transmission lines compared to Algorithm \ref{algo:islanding} (see Table~\ref{table:edges-removed}). 
Furthermore, the islanding strategies given by MILP-G and MILP-M disconnect some buses from the islands, whereas our approach ensures that each bus is connected within one island.
We note that ActivSg 500-bus system contains a collection of star subgraphs. 
To improve the run time of Algorithm \ref{algo:islanding}, we can view each star subgraph as one node, and implement Algorithm \ref{algo:islanding} on this reduced graph to approximate the islanding strategy.

Finally, we evaluate how the choices of $\alpha_1,\alpha_2,\alpha_3$ impact the controlled islanding strategy.
Our choices of $\alpha_1,\alpha_2,\alpha_3$ assign higher weight to the generator coherency.
The controlled islanding strategy did not change as we increase the values of $\alpha_2$ and $\alpha_3$.
However, by further increasing $\alpha_1$ to be $\alpha_1>3$, we observe that the controlled islanding strategy tripped more transmission lines to improve generator coherency.


\begin{table*}[ht]
\caption{This table summarizes the amount of load shedding, value of the objective function $\Bar{F}(E,\Lambda)$, and run time for IEEE 118-bus, IEEE 300-bus, ActivSg 500-bus, and Polish 2383-bus systems by using MILP-G (MILP using Gurobi), MILP-M (MILP using intlinprog), and Algorithm~1 (hybrid submodular approach) to generate $m\in\{2,3\}$ islands. ``NA" in the table represents that the baseline approaches run out of memory for Polish 2383-bus system.}\label{table:results2island}
\centering
\scalebox{0.82}{
\begin{tabular}{|c|c|ccc|ccc|ccc|}
\hline
\multirow{2}{*}{\textbf{Test Case}}      & \multirow{2}{*}{\textbf{\begin{tabular}[c]{@{}c@{}}Number \\ of Islands\end{tabular}}} & \multicolumn{3}{c|}{\textbf{Amount of Load Shedded}}      & \multicolumn{3}{c|}{\textbf{Objective Function Value}}             & \multicolumn{3}{c|}{\textbf{Run Time}}                    \\ \cline{3-11} 
                                          &                                                                                        & \multicolumn{1}{c|}{\textbf{MILP-G}} & \multicolumn{1}{c|}{\textbf{MILP-M}} & \textbf{Algorithm 1} & \multicolumn{1}{c|}{\textbf{MILP-G}} & \multicolumn{1}{c|}{\textbf{MILP-M}}         & \textbf{Algorithm 1}  &  \multicolumn{1}{c|}{\textbf{MILP-G}} & \multicolumn{1}{c|}{\textbf{MILP-M}} & \textbf{Algorithm 1} \\ \hline
\multirow{2}{*}{\textbf{IEEE 118-bus}}    & 2   & \multicolumn{1}{c|}{1691.00 MW}                                                                                    & \multicolumn{1}{c|}{1691.00 MW}    & {\bf 1526.00 MW}         & \multicolumn{1}{c|}{$4.9942\times 10^{8}$}   & \multicolumn{1}{c|}{$4.9942\times 10^{8}$} & $\mathbf{4.9941\times 10^{8}}$ & \multicolumn{1}{c|}{$\mathbf{2.35}$} & \multicolumn{1}{c|}{52.96 s}       & 101.47 s             \\ \cline{2-11} 
                                          & 3                                                    & \multicolumn{1}{c|}{1562.00 MW}                                   & \multicolumn{1}{c|}{1562.00 MW}    & {\bf 1497.00 MW}     & \multicolumn{1}{c|}{$2.3598\times 10^{7}$}       & \multicolumn{1}{c|}{$2.3598\times 10^{7}$} & $\mathbf{2.3587\times 10^{7}}$ & \multicolumn{1}{c|}{$\mathbf{2.86}$ s} & \multicolumn{1}{c|}{35.48 s}       & 154.99 s             \\ \hline
\multirow{2}{*}{\textbf{IEEE 300-bus}}    & 2                                                                                     & \multicolumn{1}{c|}{16410.95 MW}  & \multicolumn{1}{c|}{\bf 16267.45 MW}   & {\bf 16267.45 MW }     & \multicolumn{1}{c|}{$6.446\times 10^{5}$}     & \multicolumn{1}{c|}{$6.3456\times 10^{5}$} & $\mathbf{6.3431\times 10^{5}}$ & \multicolumn{1}{c|}{$\mathbf{7.94 s}$} & \multicolumn{1}{c|}{7205.2 s}      & $92.03$ s              \\ \cline{2-11} 
                                          & 3                                                                                   & \multicolumn{1}{c|}{16571.95 MW}    & \multicolumn{1}{c|}{16571.95 MW}   & {\bf 16267.45 MW}   & \multicolumn{1}{c|}{$7.7788\times 10^{5}$}        & \multicolumn{1}{c|}{$7.7847\times 10^{5}$} & $\mathbf{7.5603\times 10^{5}}$ & \multicolumn{1}{c|}{$\mathbf{35.28 s}$}  & \multicolumn{1}{c|}{7205.9 s}      & 366.32 s             \\ \hline
\multirow{2}{*}{\textbf{ActivSg 500-bus}}    & 2              & \multicolumn{1}{c|}{6277.30 MW}                                                                         & \multicolumn{1}{c|}{6277.30 MW}   & 6277.30 MW      & \multicolumn{1}{c|}{$3.1670\times 10^{8}$}     & \multicolumn{1}{c|}{$3.1670\times 10^{8}$} & $\mathbf{3.1665\times 10^{8}}$ & \multicolumn{1}{c|}{16.34 s}  & \multicolumn{1}{c|}{$\mathbf{10.71}$ s}      & 130.82 s              \\ \cline{2-11} 
                                          & 3      & \multicolumn{1}{c|}{6357.59 MW}                                                                                 & \multicolumn{1}{c|}{6357.59 MW}   & {\bf 6320.49 MW}      & \multicolumn{1}{c|}{$2.0354\times 10^{6}$}     & \multicolumn{1}{c|}{$2.0284\times 10^{6}$} & $\mathbf{2.0152\times 10^{6}}$ & \multicolumn{1}{c|}{$\mathbf{118.54}$ s} & \multicolumn{1}{c|}{7217.2 s}      & 208.78 s             \\ \hline
\multirow{2}{*}{\textbf{Polish 2383-bus}} & 2      & \multicolumn{1}{c|}{NA}                                                                                 & \multicolumn{1}{c|}{NA}            & \textbf{7511.15 MW}  & \multicolumn{1}{c|}{NA}          & \multicolumn{1}{c|}{NA}                    & $\mathbf{1.7234\times 10^{6}}$ & \multicolumn{1}{c|}{NA}  & \multicolumn{1}{c|}{NA}            & $\mathbf{335.80}$ s             \\ \cline{2-11} 
                                          & 3                & \multicolumn{1}{c|}{NA}                                                                       & \multicolumn{1}{c|}{NA}            &          \textbf{7579.65   MW}        & \multicolumn{1}{c|}{NA}   & \multicolumn{1}{c|}{NA}              &            $\mathbf{8.9825\times 10^{6}}$    & \multicolumn{1}{c|}{NA}              & \multicolumn{1}{c|}{NA}            &       $\mathbf{1328.41}$ s               \\ \hline
\end{tabular}
}
\end{table*}


\section{Conclusion}\label{sec:conclusion}

In this paper, we investigated the problem of computing a controlled islanding strategy for large-scale power systems.
We formulated the problem by taking power flow disruption, generator coherency, post-islanding stability, and blackstart generator allocation into consideration.
We translated the formulated controlled islanding problem into a matroid optimization program.
We presented the concept of hybrid submodularity, and proved that the metrics considered are hybrid supermodular.
Based on this insight, we developed an efficient local search algorithm with $\frac{1}{2}$-optimality guarantee.
We compared our solution approach with a baseline using mixed-integer linear program on four test cases including IEEE 118-bus, IEEE 300-bus, ActivSg 500-bus, and Polish 2383-bus systems.
The proposed solution approach found an islanding strategy for each test case that could outperform the baseline approach.
Moreover, our proposed approach scaled well to large cases such as Polish 2383-bus system, while the baseline did not return a result due to the problem size.

\bibliographystyle{IEEEtran}
\bibliography{mybib}

\appendix

In what follows, we introduce generator coherency and the procedure to compute the coherence matrix $\Bar{A}$.
Our discussion is based on \cite{chow1982time}.

We denote the rotor angle of generator $i$ at a steady state operating point as $\delta_i$.
Let $\Delta\delta_i$ be the rotor angle deviation of generator $i$ from the steady operating point.
Let $\Delta \delta=[\Delta\delta_1,\ldots,\Delta\delta_{|\mathcal{G}|}]^T$, where $\Delta\delta_i$ is the deviation of rotor angle of generator $i$ from the operating point.
We assume that the generators follow the linearized swing equation \cite{chow1982time} given as
\begin{equation*}
    \Delta \Ddot{\delta} = M^{-1}K\Delta \delta,
\end{equation*}
where $M$ is the inertia matrix, and $K$ is defined as
\begin{equation}
    K_{ij}=\begin{cases}
    -V_iV_jB_{ij}\cos(\delta_i-\delta_j),&\mbox{ if }i\neq j\\
    -\sum_{k\neq i}^n K_{ij},&\mbox{ if }i= j
    \end{cases}.
\end{equation}
Here $V_i$ is the per unit voltage behind transient reactance of generator $i$ and $B_{ij}$ is the imaginary part of the admittance.

Two generators $i$ and $j$ are said to be $\epsilon$-coherent if the maximum difference between their rotor angles is bounded by $\epsilon$ for all time.
Let $\sigma$ be the set of eigenvalues of $M^{-1}K$.
The $m$ number of eigenvalues with smallest magnitudes then represent the slowest modes of $M^{-1}K$.
Let $U$ be the eigenbasis corresponding to these $m$ eigenvalues.
Let $U_1$ be the matrix obtained by extracting the rows of $U$ that correspond to the reference generators. 
Then the coherence matrix $A$ is computed as $\Bar{A}=UU_1^{-1}$.

\begin{IEEEbiography}[{\includegraphics[width=1in,height=1.25in,clip,keepaspectratio]{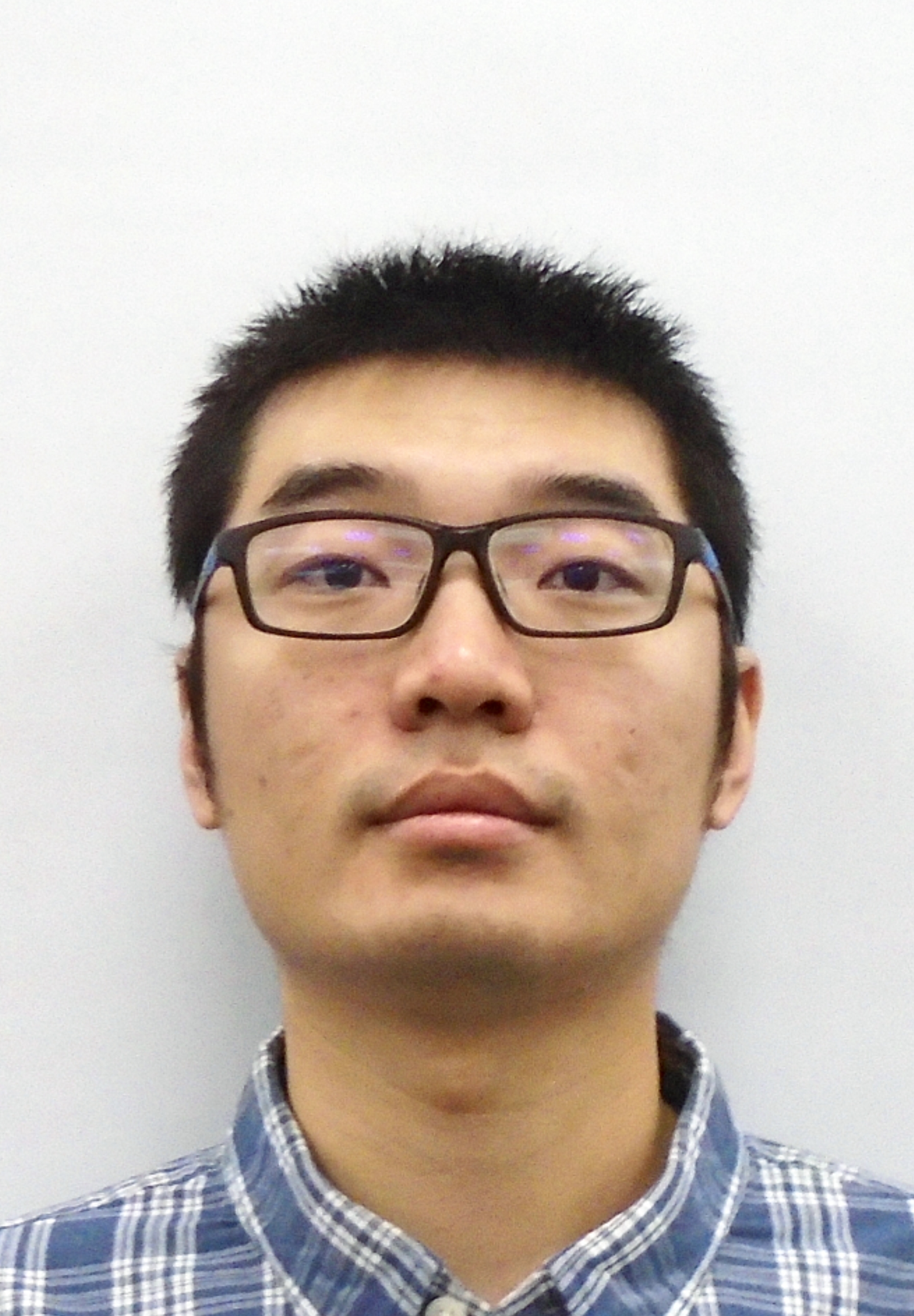}}]
{Luyao Niu}(M'22) is a postdoctoral scholar at the Network Security Lab (NSL), Department of Electrical and Computer Engineering, at the University of Washington - Seattle.
He received the B.Eng. degree from the School of Electro-Mechanical Engineering, Xidian University, Xi’an, China, in 2013. He received the M.Sc. degree and Ph.D. degree from the Department of Electrical and Computer Engineering, Worcester Polytechnic Institute (WPI) in 2015 and 2022. 
He is the author of the GameSec Outstanding Paper (2018) and was finalist for ACM/IEEE International Conference on Cyber-Physical Systems (ICCPS) 2020 Best Paper Session Award. 
His research interests include optimization, game theory, and scalable and verifiable control and security of cyber physical systems.
\end{IEEEbiography}
\vfill

\begin{IEEEbiography}[{\includegraphics[width=1in,height=1.25in,clip,keepaspectratio]{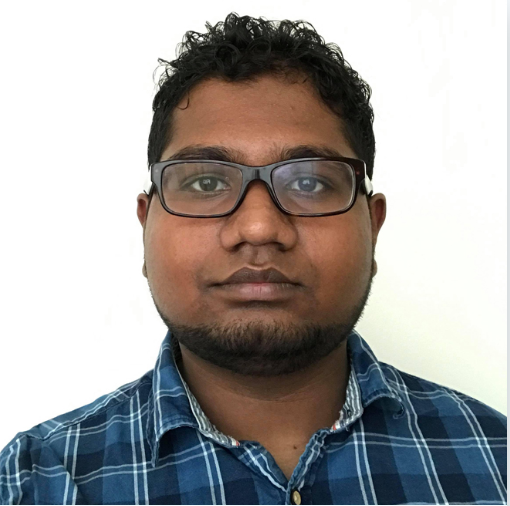}}]{Dinuka Sahabandu}(M'23) is a postdoctoral scholar  at the Network Security Lab (NSL), Department of Electrical and Computer Engineering, at the University of Washington - Seattle.
He received the B.S. degree and M.S. degree in Electrical Engineering from the
Washington State University - Pullman in 2013
and 2016, respectively. 
He received the Ph.D. degree from the Department of Electrical and Computer Engineering at the University of Washington - Seattle in 2023.
His research interests include game theory for network security and control of multi-agent systems.
\end{IEEEbiography}

\vfill
\begin{IEEEbiography}[{\includegraphics[width=1in,height=1.25in,clip,keepaspectratio]{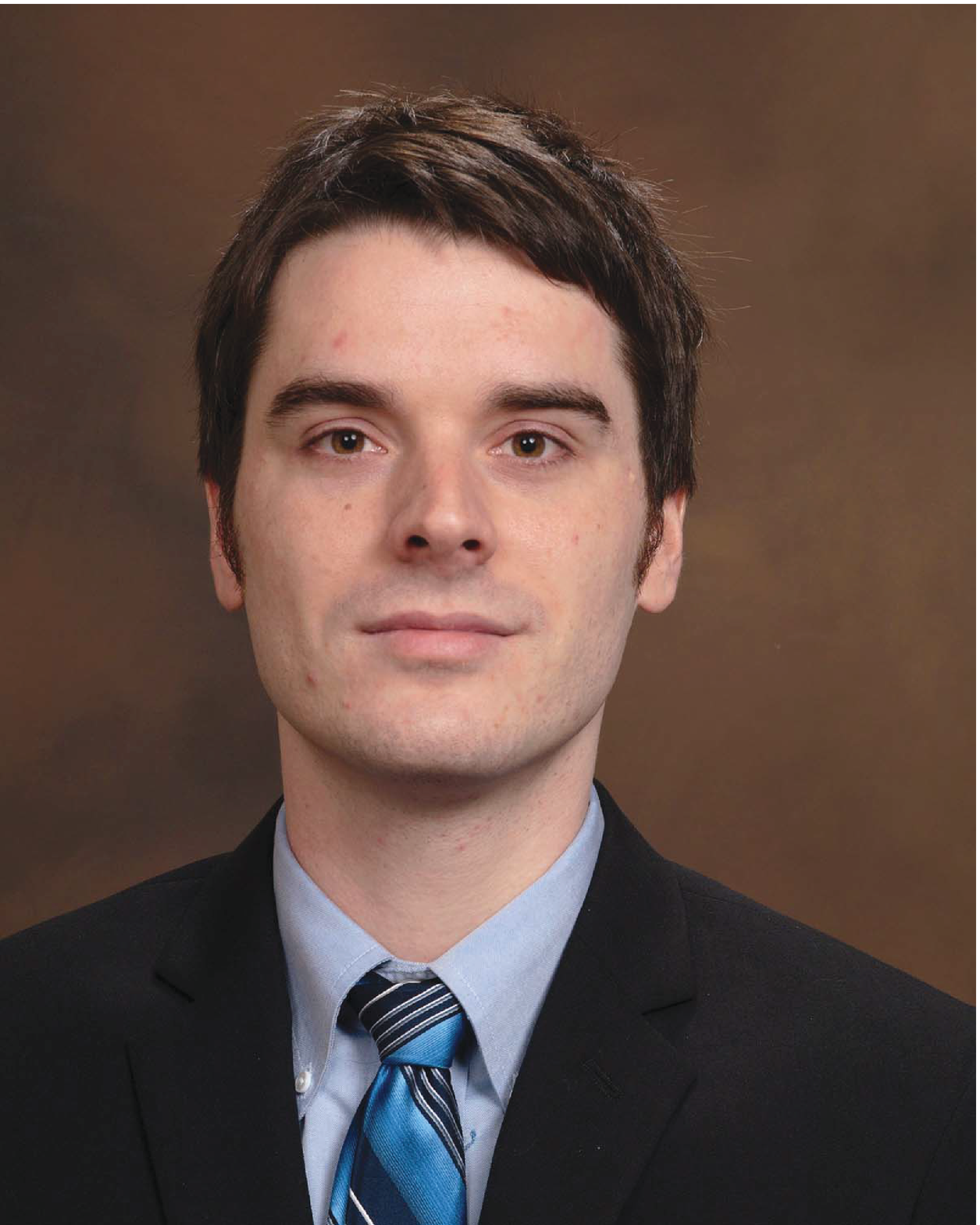}}]{Andrew Clark}(M'15)
is an Associate Professor of Electrical and Systems Engineering at Washington University in St. Louis. He received the B.S. degree in Electrical Engineering and the M.S. degree in Mathematics from the University of Michigan - Ann Arbor in 2007 and 2008, respectively. He received the Ph.D. degree from the Network Security Lab (NSL), Department of Electrical Engineering, at the University of Washington - Seattle in 2014. He is author or co-author of the IEEE/IFIP William C. Carter award- winning paper (2010), the WiOpt Best Paper (2012),
and the WiOpt Student Best Paper (2014), and was a finalist for the IEEE CDC 2012 Best Student-Paper Award. 
He received the GameSec Outstanding Paper Award (2018) and was finalist for the ACM ICCPS Best Paper Award (2016, 2018, 2020).
He won the General Motors AutoDriving Security Award at VehicleSec (2023).
He received the University of Washington Center for Information Assurance and Cybersecurity (CIAC) Distinguished Research Award (2012), Distinguished Dissertation Award (2014), an NSF CAREER award (2020), and an AFOSR YIP award (2022). His research interests include control and security of complex networks, submodular optimization, and control-theoretic modeling of network security threats.
\end{IEEEbiography}
 
\vspace{11pt}


\vfill

\begin{IEEEbiography}[{\includegraphics[width=1in,height=1.25in,clip,keepaspectratio]{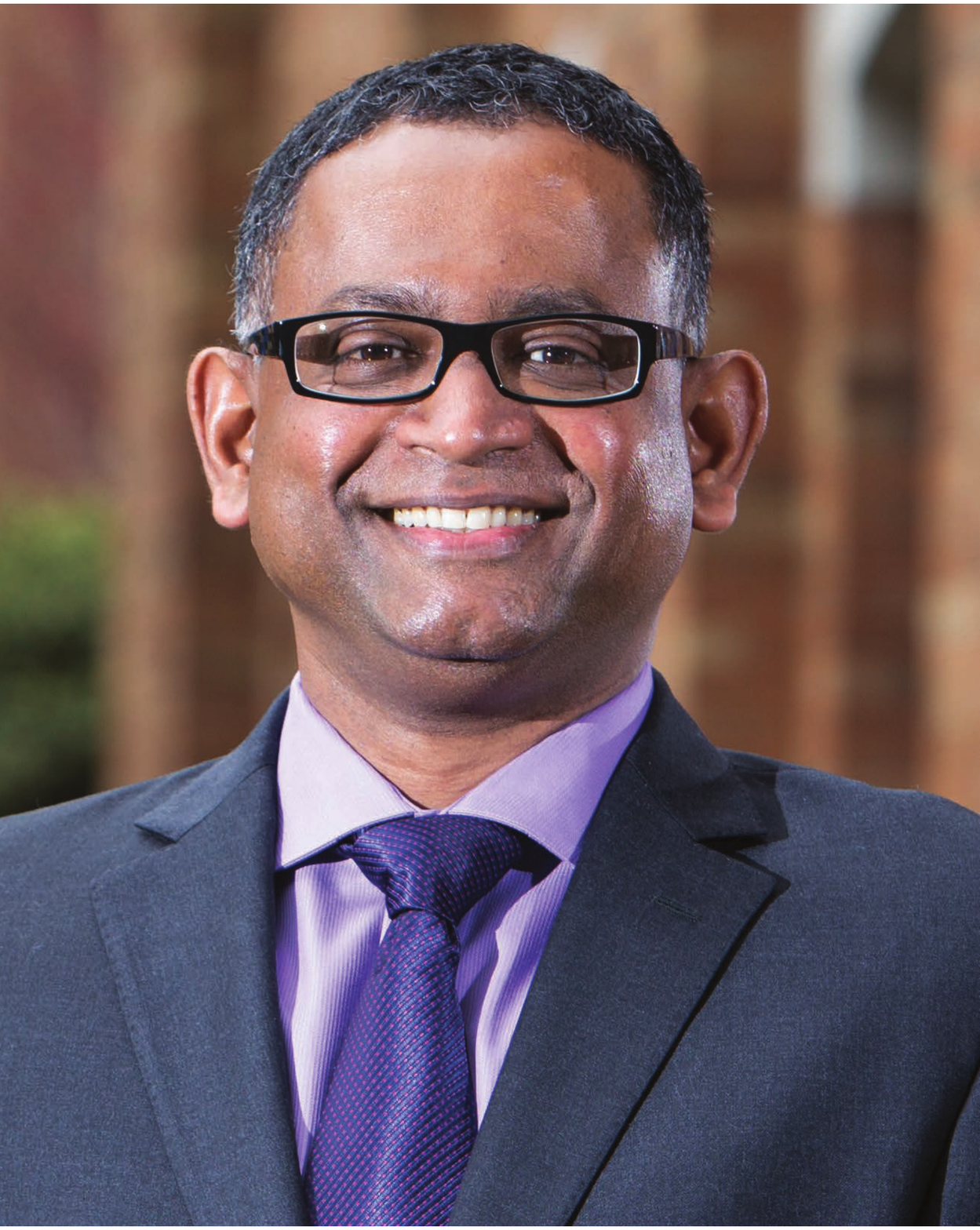}}]{Radha Poovendran} (F’15) is a Professor in
the Department of Electrical and Computer
Engineering at the University of Washington
(UW) - Seattle. He served as the Chair
of the Electrical and Computer Engineering
Department at UW for five years starting
January 2015. He is the Director of the Network
Security Lab (NSL) at UW. He is the Associate
Director of Research of the UW Center for
Excellence in Information Assurance Research
and Education. He received the B.S. degree in
Electrical Engineering and the M.S. degree in Electrical and Computer
Engineering from the Indian Institute of Technology- Bombay and
University of Michigan - Ann Arbor in 1988 and 1992, respectively.
He received the Ph.D. degree in Electrical and Computer Engineering
from the University of Maryland - College Park in 1999. His research
interests are in the areas of wireless and sensor network security,
control and security of cyber-physical systems, adversarial modeling,
smart connected communities, control-security, games-security, and
information theoretic security in the context of wireless mobile networks.
He is a Fellow of the IEEE for his contributions to security in cyberphysical
systems. He is a recipient of the NSA LUCITE Rising Star
Award (1999), National Science Foundation CAREER (2001), ARO
YIP (2002), ONR YIP (2004), and PECASE (2005) for his research
contributions to multi-user wireless security. He is also a recipient of
the Outstanding Teaching Award and Outstanding Research Advisor
Award from UW EE (2002), Graduate Mentor Award from Office of
the Chancellor at University of California - San Diego (2006), and the
University of Maryland ECE Distinguished Alumni Award (2016). He
was co-author of award-winning papers including IEEE/IFIP William C.
Carter Award Paper (2010) and WiOpt Best Paper Award (2012).
\end{IEEEbiography}






\end{document}